\documentclass[preprint,12pt,authoryear]{elsarticle}
\usepackage{amsthm,amsmath,amssymb,amsfonts,amsthm,bm,booktabs,color}
\usepackage{rotating}

\newtheorem{definition}{Definition}
\newtheorem{remark}{Remark}
\newtheorem{theorem}{Theorem}
\newtheorem{lemma}{Lemma}
\newtheorem{proposition}{Proposition}

\usepackage{mathtools}
\graphicspath{{./figuras/}}
\usepackage{ulem}

\DeclareMathOperator*{\argmax}{arg\,max}
\newcommand{\vertiff}{\rotatebox{90}{$\Leftrightarrow$}}
\begin{document}

\begin{frontmatter}
\title{On the classification problem for Poisson Point Processes}

\author{Alejandro Cholaquidis$^*$}
\author{Liliana Forzani$^\dag$}
\author{Pamela Llop$^{\dag \ddagger}$}
\author{Leonardo Moreno$^*$}

\address{$^*$Centro de Matem\'atica, Facultad de Ciencias (UdelaR)}
\address{$^{\dag}$Facultad de Ingenier\'ia Qu\'imica (UNL)}
\address{$^{\ddagger}$Instituto de Matem\'atica Aplicada del Litoral (UNL - CONICET)}

\begin{abstract}
For Poisson processes taking values in any general metric space, we tackle the problem of supervised classification in two different ways: via the classical $k$-nearest neighbor rule, by introducing suitable distances between patterns of points and via the Bayes rule, by estimating nonparametricaly the intensity function of the process. In the first approach we prove that, under the separability of the space the rule turns out to be  consistent. In the former, we prove the consistency of rule by proving the consistency of the estimated intensities. Both classifiers have shown to have a good behaviour under departures from the Poisson distribution. 
\end{abstract}


\end{frontmatter}

\section{Introduction}\label{intro}
Spatial point processes are commonly used to model the spatial structure of points formed by the location of individuals in space.  The growing interest in this kind of process is related to the wide range of areas where they can be applied. For instance, in ecology, they can be used to model the distribution of herds of animals, the spreading of nests of birds, speckles of trees or plants or the eroded areas in rivers or seas. In geography, the position of earthquakes or volcanoes can by modelled by this kind of processes. They can be also used to model the distribution of galaxies in astronomy, the locations of subscribers in telecommunications, among others. There exists a vast literature on this area, just to name a few, we refer to the recent book \textit{Spatial Data analysis in ecology and agriculture using R} (\cite{plant}), which contains many other possible applications and techniques, as well as real data examples. In \cite{Illian09}, the authors propose a hierarchical modelling of the interaction structure in the plant community. The current interest on this kind of process also appears in connection with the new developments in Functional Neuroimaging techniques (for example fMRI), where it is possible to record in real time the location of the activation zones of the brain (see for instance \cite{kang11,kang14}, and  \cite{Yarkoni10}). In this context, in order to perform classification between healthy and unhealthy people, the differences between the neurons that fire under some stimuli can be measured by modelling them as spatial Poisson processes with different intensities. In \cite{mateu15}, the authors do a review of several distances used to measure the differences between two spatial patterns in order to perform clustering or classification (see also \cite{victor:97}). In a different application area, crime modelling and mapping using geospatial technologies (which include the use of spatial point process) is, quoting \cite{leitner}, ``a topic of much interest mostly to academia, but also to the private sector and the government'', see also \cite{handcrim} and \cite{gervini}. On this topic in Section \ref{realdata} we study the spatial distribution of three different crimes which took place in Chicago between 2014 and 2016, by using an open access database containing, among other variables, the spatial location of the crimes.

\

The aim of this manuscript is to tackle the supervised classification pro\-blem for Poisson point processes by framing it in the functional data setting. In particular, we prove the consistency of the $k$-nearest neighbors classifier in a more general context by proving the separability of the space and the Besicovitch condition (see \cite{cer06} and \cite{ffl:12} for a deepest reading on this topic). Via some simulation studies, we show how different choices of distances lead to different results on the classification. In addition, following the ideas in \cite{diggle}, we also propose a nonparametric estimator of the intensity function, prove its consistency and plug it into the Bayes rule to get a consistent classifier. This last approach is similar to the one proposed in \cite{kang14} but we do not assume that the intensities vary in a parametric family. Through some simulation studies we show the good performance of the $k$-NN rule so that it can be considered as an easier to implement alternative to the classical Bayes. More precisely, the $k$-NN classifier does not require the estimation of the intensity function (which is computationally expensive) and it can be employed in more general settings. With regard to the last statement, it is important to highlight that, although most of the classical applications of spatial point processes are for recorded locations in $\mathbb{R}^2$ or  $\mathbb{R}^3$, we do not restrict our approach to that case, allowing the realizations of the processes to live in a general metric space (as functional metric spaces, Riemannian manifolds, among others).

\

The manuscript is outlined as follows: in Section \ref{defnot} we present definitions and preliminary results that we will use throughout the work. Section \ref{secc:intensidad} is devoted to introduce an estimator of the intensity of the process in order to plug it in the Bayes rule an prove its consistency. In Section \ref{secc:besi} we handle the problem of choosing a suitable distance to guarantee the separability of the space and the Besicovitch condition, in order to get the consistency of the $k$-NN estimator. Section \ref{pqbesi} is devoted to the study of the metric dimension of the space introduced in the Section \ref{secc:besi}. In Section \ref{gibss} we extend the results to a more general class of processes: the Gibss processes. In Section \ref{simu} we perform some simulation studies in order to asses the performance of the classification rules for different scenarios as well as to see the effect of changing some parameters in the estimation and robustness when the model is not Poisson. Finally in Section \ref{realdata} we perform classification in a real data scenario. All the proofs are given in the Appendix.

\section{Definitions and preliminary results} \label{defnot}

This section is devoted to introduce some definitions and tools we will use throughout the  paper. We will start with the definition of the main object of this paper, the Poisson point process and then we will turn to classification rules in our context. For a deeper read on Poisson processes we refer to  \cite{gg:10}, \cite{k:93} and \cite{mw04}.

\subsection{Poisson process}\label{pp}
Let $(S,\rho)$ be a separable and bounded set metric space, endowed with a Borel measure $\nu$, let us denote by $\mathcal{B}(S)$ the Borel $\sigma$-algebra on $S$ and by $S^\infty$ the set of elements (subsets) $x$ of $S$ whose cardinal, $\# x$, is finite. This is, 
\[ 
S^\infty \doteq \{ x\subset S: \# x < \infty \}.  
\]
Let $\lambda:S\rightarrow \mathbb{R}^+$ be an integrable function. Given a probability space $(\Omega,\mathcal{A},\mathbb{P})$,  we will say that a function $X:\Omega\rightarrow S^\infty$ is a \textbf{Poisson process} on $S$ with intensity $\lambda$ (we will denote $X\sim Poisson(S,\lambda)$) if:
\begin{itemize} 
\item the functions $N_A:\Omega\rightarrow \{0,\dots,\infty\}$ defined as $N_A(\omega)=\#(X(\omega)\cap A)$ are random variables for all $A\in \mathcal{B}(S)$;
\item given $n$ disjoint Borel subsets $A_1,\dots,A_n$ of $S$, the random variables $N_{A_1},\dots,N_{A_n}$ are independent;
\item $N_A$ follows a Poisson process with mean $\mu(A)$ (we will denote $N_A\sim \mathcal{P}(\mu(A))$, being $$\mu(A)=\int_A \lambda(\zeta)d\nu(\zeta).$$
\end{itemize}

Let $\mathcal{S}^\infty= 2^{{S}^\infty}$ be the $\sigma$-algebra of part of $S^\infty$. If  $X$ is a Poisson process, the distribution $P_X$ of $X$ on $\mathcal{S}^\infty$ is defined as $P_X(B)=\mathbb{P}(X\in B)$ for $B \in {\mathcal S}^\infty$.

\

A well-known result (see \cite{mw04}) on point processes states that, if $X_1$ and $X_2$ are Poisson processes with intensity $\lambda_1$ and $\lambda_2$, respectively, with values on a non-empty bounded metric space $(S,\rho)$ such that $\mu_i(S)<\infty$, $i=1,2$, the distribution of $X_1$ is absolutely conti\-nuous with respect to the distribution of $X_2$ ($P_{X_1}\ll P_{X_2}$) with Radon Nikodym derivative
\[
f_{X_1}(x)=\exp\Big[ \mu_2(S)-\mu_1(S)\Big] \prod_{\xi \in x} \frac{\lambda_1(\xi)}{\lambda_2(\xi)},
\]  
with $0/0=0$. As a consequence observe that if $X_2\sim Poisson(S,1)$ then, for all $X\sim Poisson(S,\lambda)$, $P_{X}\ll P_{X_2}$ and 
\begin{equation} \label{eqdens2}
f_X(x)=\exp\Big[ \nu(S)-\mu(S)\Big] \prod_{\xi \in x} \lambda(\xi),
\end{equation}  
where $\mu(S)=\int_S \lambda d\nu$. 

\subsection{Classification} \label{clasifbayes}
Given a set $\{(X_i,Y_i)\}_{i=1}^n \in S^\infty \times \{0,\ldots,M\}$ of iid pares with the same distribution as $(X,Y)$, the aim of classification is, given a new observation $X$,  to predict the class $Y$ to which $X$ belongs. In this context, a classification rule is a measurable function $g:S^\infty \rightarrow \{0,\ldots,M\}$ which, for a new observation $X$, returns a  label $Y\in  \{0,\ldots,M\}$. It was shown (see, e.g., \cite{dev96}) that the optimal classifier is the \textbf{Bayes rule} $ g^*$ which minimizes the probability of error or, equivalentely, which maximizes the posterior probabilities:
\[
 g^*(x)=\argmax_{g:S^\infty \rightarrow \{0,\ldots,M\}} {\mathbb P}(g(X)\neq Y).
\]
The mimimun probability  of error $L^*={\mathbb P}(g^*(X)\neq Y)$ is known as \textbf{Bayes error}. If $L_n={\mathbb P}(g_n(X)\neq Y|D_n)$ is the probability of error of a sequence $g_n$ of classifiers built up from a training sample, it is said that the sequence is \textbf{weakly consistent} if $L_n$ converges in probability to $L^*$ as $n\to\infty$.

\

In our context, we assume that $X$ conditioned to $Y$ has Poisson distribution therefore, following (\ref{eqdens2}), in Lemma \ref{bayes} we obtain an expression for the Bayes rule as a function of the intensities of the processes. 
\begin{lemma}\label{bayes}
Let $(X,Y)\in S^\infty \times \{0,\ldots,M\}$. Let $X_j\doteq X|Y=j$ be Poisson processes on $S^\infty$ with intensities $\lambda_j$, $j=1,\ldots,M$, respectively. Therefore, the Bayes rule classifies a point $x\in S^\infty$ in class $j$ if 
\begin{equation} \label{eqclasif}
\exp\Big[\mu_i(S)-\mu_j(S)\Big]\prod_{\xi \in x}  \frac{\lambda_j(\xi)}{ \lambda_i(\xi)} > \frac{p_i}{p_j}, \quad \forall \, i\neq j, 
\end{equation}
where  $p_i=\mathbb{P}(Y=i)$, $i=1,\ldots,M$ and as before, $\mu_i(S)=\int_S \lambda_i(\zeta)d\nu(\zeta)$, $i=1,\ldots,M$.
\end{lemma}

Observe that, in order to apply the Bayes rule we will need to estimate the intensities $\lambda_j$ of the processes which will be done in Section \ref{secc:intensidad}.

\

Another well known classification rule is the $\bm k$\textbf{-nearest neighbor} rule which, in our context will classify a point $x\in S^\infty$ in class $j$ if,  for all $i\neq j$,
\[
\sum_{k=1}^n w_{nk} {\mathbb I}_{\{Y_k = j\}} > \sum_{k=1}^n w_{nk} {\mathbb I}_{\{Y_k = i\}},
\] 
where the weights $w_{nk}$ are $1/k$ for the $k$-nearest neighbors of $x$ and $0$ elsewhere. We say that $X_i$ is the $k$-nearest neighbor  $X$ among $\{X_1,\ldots, X_n\}$ if the distance $d(X_i,X)$ is the $k$-th smallest among $d(X_1,X), d(X_2, X), \ldots, d(X_n,X)$. 

\

For random variables taking values in a finite dimensional space (for instance ${\mathbb R}^d$), it is well-known (see \cite{sto77}) that the $k$-NN rule is $L^2$-universally consistent  provided that $k\to\infty$ and $k/n\rightarrow 0$. However, when they take values in infinite dimensional spaces (as in this case), the consistency is not necessarily true (even weakly than $L^2$) as it was studied by \cite{cer06}. Nevertheless, \cite{ffl:12} gave sufficient conditions to ensure $L^2$-consistency  of the classical estimators of the regression function $\eta(x)=\mathbb{E}(Y|X=x)$. That conditions are the separability of the metric space $(S^\infty,d)$ for a given metric $d$ and \textit{Besicovitch condition}, which can be stated as:  
\begin{equation} \label{besicovitch2}
\lim_{\epsilon \rightarrow 0} P_X \Big\{ x: \lim_{\delta\rightarrow 0} \frac{1}{P_X \big(B_d (x,\delta)\big)}
\int_{B_d(x,\delta)} |\eta(x)-\eta(y)|d P_X (y) > \epsilon \Big\} =0,
\end{equation}
 It is immediate that $P_X$-a.s. continuity of $\eta$ is a sufficient condition for (\ref{besicovitch2}).
In order to get the consistency of the $k$-NN rule in the context of Poisson processes, we will study  in Section \ref{secc:besi} the problem of choosing a suitable distance $d$ which leads the separability of the space $(S^\infty,d)$ and the Besicovitch condition for  $\eta(x)$.

\section{Main Results} \label{ALE}
Throughout all this section we will assume that $(S,\rho)$ is a separable compact metric space.

\subsection{Bayes rule in the context of Poisson processes: consistency and the estimation of the intensity}\label{secc:intensidad}

In this section we propose to estimate nonparametrically the intensity functions $\lambda_j$ $j=1,\ldots, M$ in order to plug in them in Equation (\ref{eqclasif}) to get the Bayes rule for Poisson processes. Following \cite{diggle}, given a realization $\{\xi_1,\dots,\xi_n\}$ of the process $X$ with values in $S^\infty$, we estimate the intensity $\lambda(\zeta)$ of $X$ in the point $\zeta \in S$ as
\begin{equation} \label{estint}
\hat{\lambda}(\zeta)=\frac{1}{K_\sigma(\zeta)}\sum_{i=1}^n \frac{1}{\sigma^d} k\Bigg(\frac{\rho(\zeta,\xi_i)}{\sigma}\Bigg),
\end{equation}
with $k:\mathbb{R}\rightarrow \mathbb{R}^+$ a symmetric, no negative kernel, $\sigma>0$ a smoothing parameter and
$$
K_\sigma(\zeta)=\int_S \frac{1}{\sigma^d}k\Bigg(\frac{\rho(\zeta,\xi)}{\sigma}\Bigg)d\nu(\xi).
$$
Given a random sample of Poisson processes $X_1,\dots,X_m$, each with realization  $X_j=\{\xi_1,\dots,\xi_{n(j)}\}$, $j=1,\dots,m$, we define an estimator of the intensity $\lambda(\zeta)$ by
\begin{equation} \label{estint1}
\hat{\hat{\lambda}}_m(\zeta)=\frac{1}{m}\sum_{j=1}^m \hat{\lambda}_j(\zeta),
\end{equation}
where $\hat{\lambda}_j(\zeta)$ is an estimation of $\lambda(\zeta)$ as given in (\ref{estint}) for the realization $X_j=\{\xi_1,\dots,\xi_{n(j)}\}$, $j=1,\dots,m$. 

{Via simulation study (see Section \ref{sec:efectoh}), we observed that when performing classification by using the Bayes rule as in (\ref{eqclasif}) with estimated intensities (\ref{estint1}) the best choice is $\sigma_1=\sigma_2$.}

In the following theorem we show the consistency of the estimator  given in (\ref{estint1}). 
\begin{theorem}\label{consistencia} Let us assume that the intensity function {$\lambda$} is continuous and that, for all $\zeta \in S$, there exists $\lambda_0>0$ such that $\lambda(\zeta)\geq \lambda_0$. Let $k:\mathbb{R}\rightarrow \mathbb{R}^{+}$ be a symmetric, no negative continuous kernel such that $supp(k)\subset [0,diam(S)]$  and $k(x)>k_0>0$ $\forall x\in [0,diam(S)]$. Then, for almost all $\zeta\in S$ (w.r.t. $\nu$), there exists $\sigma_m(\zeta)\rightarrow 0$, such that
\begin{equation} \label{eqth2}
\lim_{m\rightarrow \infty} \Big|\hat{\hat{\lambda}}_m(\zeta)-\lambda(\zeta)\Big|=0 \quad \text{ a.s.}
\end{equation}
\end{theorem}
\begin{remark} It is easy to see that if $S\subset\mathbb{R}^d$, $\nu$ is absolutely continuous w.r.t. the Lebesgue measure, the density is bounded away from 0 and $S$ is standard (see Definition 1 in \cite{cue04}), then $(\log(m)/m)^{1/(2d)}/\sigma_m\rightarrow 0$ and $\sigma_m\rightarrow 0$ is enough to get \eqref{eqth2}.%since in that case $\gamma_m\sim \sigma_m^d$. 
\end{remark}
For a recent review on the estimation of the intensity function for general point process see \cite{vliesh:12}.

\subsection{$k$-NN rule in the context of Poisson process: consistency}\label{secc:besi}

As we said in Section \ref{clasifbayes}, in order to get the separability of $S^\infty$ as well as the Besicovitch condition, we need to chose a suitable distance. Since the elements of $S^\infty$ are subsets of $(S,\rho)$, a quite natural choice is  the Hausdorff distance, which measure how far two subsets of a metric space are from each other. It is defined as follows.

\begin{definition} 
Given two non-empty compact sets $A,C\subset S$, the \textbf{Hausdorff distance} between them is defined by
\[
d_H(A,C)=\max\Big\{\sup_{a\in A} d(a,C),\ \sup_{c\in C}d(c,A)\Big\},
\]
where $d(a,C)=\inf\{\rho(a,c):c\in C\}$. 
\end{definition}
\begin{remark}
Observe that, when $S$ is bounded the metric $d_H$ on $S^\infty$ is well defined since in this case $\#x<\infty$ for all $x\in S^\infty$. 
\end{remark}
In the following two propositions we state the sufficient conditions to get the $L^2$-consistency of the $k$-nearest neighbor rule in $(S^\infty,d_H)$. 
\begin{proposition} \label{separ} 
The space $(S^\infty,d_H)$ is separable. 
\end{proposition}

\begin{remark} \label{haus} 
Moreover if $S$ is complete, the metric space of compact non-empty subsets of $S$ endowed with $d_H$ turns out to be a complete and locally compact metric space (see Chapter 4 in \cite{roc09}). 
\end{remark}
\begin{proposition} \label{besifuerte} Let us consider $(X,Y)\in S^\infty \times \{0,\ldots,M\}$. Suppose that $X_j\doteq X|Y=j$ {is} a Poisson process on $S^\infty$ with intensity function $\lambda_j$, for $j=1,\ldots,M$, respectively. Let us assume that $\lambda_j$, are continuous functions of $\rho$ and that the measure $\nu$ does not have atoms (i.e. $\nu(\{\zeta\})=0$ for all $\zeta\in S$). Then, for all $x\in S^\infty$ condition (\ref{besicovitch2}) holds for $\eta(x)=\mathbb{E}(Y|X=x)$ with $d=d_H$. 
\end{proposition}
From Propositions \ref{separ} and \ref{besifuerte} and Theorems 4.1 and 5.1 in \cite{ffl:12} it follows the consistency of the $k$-NN estimator of the regression function $\mathbb{E}(Y|X=x)$  which in turn gives the consistency of the classification rule built up from such estimator.

\

Although we stated the consistency of the $k$-NN rule for the Hausdorff distance,
we could have two points very close in Hausdorff distance but with very dissimilar cardinal. Via some simulation studies we noted that adding to the Hausdorff distance a term that forces points close enough in Hausdorff distance to have the same cardinality, the performance of the classification rule improves considerably. Basically this is due to the fact that in point processes analysis the cardinality of the points is an important characteristic to distinguish between populations. {Moreover, we performed a simulation study (see Section \ref{sec:integrenigual}) to show that, for two populations with the same expected number of points but different intensity, the Haussdorf distance is still a good choice, without the necessity of adding a new term. With all this in mind, we define new metrics in $S^\infty$ which have shown to outperform Hausdorff distance, and give the consistency of the $k$-NN rule for them. 
\begin{definition} \label{distpoints} Given $x,y\in S^\infty$, we define a new distance $d$ on $S^\infty$ as:
\begin{equation} \label{dist}
d(x,y)=\frac{1}{\text{diam}(S)}d_H(x,y)+d_0(x,y),
\end{equation}
where $diam(S)$ denotes the diameter of $S$ (i.e $diam(S)=\sup_{x,y\in S} \rho(x,y)$) and $d_0:S^\infty\times S^\infty\rightarrow [0,1]$ is a function (not necessarily a distance) which verifies:
\begin{itemize}
\item[1.] $\#x=\#y$ implies $d_0(x,y)=0$;
\item[2.] $d_0(x,y)=d_0(y,x)$;
\item[3.] for all $z\in S^\infty$, $d_0(x,z)\leq d_0(x,y)+d_0(y,z)$;
\item[4.] $\forall \, x\in S^\infty \, \text{there exists } \epsilon_0=\epsilon_0(x)>0 \text{ such that, if } d_0(x,y) < \epsilon_0 \text{ then } \#x = \#y.$
\end{itemize}
\end{definition}

\

In what follows we list a set of functions verifying conditions 1--4 in Definition \ref{distpoints}: 
\begin{itemize}
\item $d_0(x,y)= \frac{|\#x-\#y|}{1+|\#x-\#y|}$;
%\item $d_0(x,y)=1-\exp(-\big|\#x-\#y\big|)$;
\item Hellinger: $d_0(x,y)^2=1-\exp\Big\{-\frac{1}{2}\big(\sqrt{\#x}-\sqrt{\#y}\big)^2\Big\}$;
\item Kulback Leibel: $d_0(x,y)=1-\exp\Big\{(\#y-\#x) \log \left( \#x / \#y\right)\Big\}$.
\end{itemize}
As before, in the following two propositions we state sufficient conditions to get the $L^2$-consistency (see Theorem 4.1 and 5.1 in \cite{ffl:12})) of the $k$-nearest neighbor rule in $(S^\infty,d)$ with $d$ as in (\ref{dist}), {which in turn gives the consistency of the classification rule built up from such estimator.}

\begin{proposition} \label{separ2} 
The space $(S^\infty,d)$ with $d$ as defined in (\ref{dist}) is separable.
\end{proposition}
\begin{proposition} \label{besi2}
Let $(S,\rho)$ be a bounded metric space. If the intensity $\lambda$ of a Poisson process $X$ defined on $S^\infty$ is  continuous on $S$ (with respect to the distance $\rho$) then the regression function $\eta $ is continuous with respect to the metric $d$ defined in (\ref{dist}) and then it fulfils condition (\ref{besicovitch2}). 
\end{proposition}

As we will see in Section \ref{sec:efectok}, for all the distances $d_0$ before defined, higher number of neighbor gives better classification (although seven neigh\-bors could be the right choice since for seven or more neighbors the results are the same).

\section{Why do we need to prove the Besicovitch condition?}\label{pqbesi}
The Besicovitch condition would be trivial if the space $S^{\infty}$ were finite dimensional. However, this space is not even a vector space therefore, first we need to define what ``infinite dimensional" means. 

\begin{definition} \label{nagata} A metric space $(\mathcal{H},d)$ is finite dimensional (in the Nagata sense) if there exists $n_0>0$ such that for all $a\in X$, $r>0$, and $n>n_0$ points $y_i\in B_d(a,r)$, there exists $i\neq j$ such that $d(y_i,y_j)\leq r$. A metric space is said to be $\sigma$-finite dimensional if it is equal to the numerable union of finite dimensional sets.\end{definition}

The following result ensures that $(S^\infty,d)$ where $d$ is as in (\ref{dist}) is not finite dimensional.  

\begin{proposition} \label{diminfiprop} 
Let us assume that there exists $r_0>0$ such that for all $\zeta\in S$, and all $r<r_0$, $\partial B(\zeta,r)$ contains two points  $\pi^1,\pi^2$ such that $\rho(\pi_1,\pi_2)=2r$. Then the space $(S^\infty,d\big)$ is not finite dimensional. 
\end{proposition}

\section{Extensions of the results to Gibbs process}\label{gibss}
Gibbs processes appear as a natural generalization of Poisson processes since they allow a spatial dependency between the numbers of points in two disjoints subsets of $S$ (compare with the definition of Poisson process introduced in Section \ref{pp}). We prove that Proposition \ref{besifuerte} can be extended to this class of processes which, for instance, are being used in telecommunications to model the position of base stations for improving the performance of a wireless network, see \cite{zhu:15}, \cite{kelif} and \cite{guo}. 

Recall that a process is Gibbs if its density with respect to $Poisson(S,1)$ has the form $f(x)=c\exp(-U)$
being $c$ constant, where the energy $U(x)$ is admissible, in the sense that satisfy:
$$\sum_{n=0}^\infty \frac{e^{-\nu(S)}}{n!}q_n<\infty\quad \text{and} \quad q_n=\int_{S^n}\exp(-U(x))d\nu(x_1)\dots d\nu(x_n)<\infty.$$
Since we will assume that $S$ is compact, $\nu(S)<\infty$ and $U$ is a bounded function, the admissibility condition will be fulfilled. We will assume that the energy is of the form (see pg. 95 in \cite{gg:10})
\begin{equation} \label{energy}
U(x)=\sum_{i=1}^n\varphi(x_i)+\sum_{i=1}^n\sum_{j>i}^n \psi(\|x_i-x_j\|).
\end{equation}
This includes as a particular case $Poisson(S,\lambda)$, taking $-\varphi(x)=\log(\lambda(x))$ and $\psi=0$, and \textit{Strauss processes}, for which $f(x)=c\beta^{n(x)}\gamma^{s_r(x)}$ with $s_r(x)=\sum_{i<j}\mathbb{I}_{\{\|x_i-x_j\|<r\}}$ and $n(x)=\#x$. The following proposition extends  Proposition \ref{besifuerte} to this kind of process.

\begin{proposition} \label{besifuerte2} Let us consider $(X,Y)\in S^\infty \times \{0,\dots,M\}$ being $(S,\rho)$ compact. Suppose that for all $j=1,\dots,M$, $X_j\doteq X|Y=j$ is a Gibbs process  with energy $U_j$ given by  \eqref{energy}, with $\varphi_j$ and $\psi_j$ continuous functions. Let us assume that the measure $\nu$ does not have atoms (i.e. $\nu(\{\zeta\})=0$ for all $\zeta\in S$). Then, for all $x\in S^\infty$ the Besicovitch condition holds for $\mathbb{E}(Y|X=x)$ with $d=d_H$. 
\end{proposition}

%%%%%%%%%%%%%%%%%%%%%%%%%%%%%%%%%%%%%%%%%%%%%%%%%%%%%%%%
\section{Simulations} \label{simu}

In order to assess the performance of the proposed classification rules for two populations and see how the nature of the density function affect the methods we have implemented some simulation studies. First we show the behaviour in three different scenarios, one in which the densities are smooth and decrease to zero exponentially fast, another for very wiggly densities, and a last one where the expected number of points is the same, but the distribution of points is different. We also carry out three simulation studies to show the robustness under departure from the Poisson assumption, the effect of $\sigma$ in the estimation of the intensities (\ref{estint}) and the effect of $k$ in the $k$-nearest neighbor distances (\ref{dist}). 

\

In what follows, we will use the following notation for the different distances:
\begin{itemize}
\item \textsf{KNN\_Hausdorff}: $k$-NN in $(S^\infty,d_H)$;
\item \textsf{KNN\_Hausdorff\_d1}: $k$-NN in $(S^\infty,d)$ with $d$ given in (\ref{dist}) and $d_0(x,y)= \frac{|\#x-\#y|}{1+|\#x-\#y|}$;
\item \textsf{KNN\_Hausdorff\_Hellinger}: $k$-NN in $(S^\infty,d)$ with $d$ given in (\ref{dist}) and  
$d_0(x,y)^2=1-\exp\Big\{-\frac{1}{2}\big(\sqrt{\#x}-\sqrt{\#y}\big)^2\Big\}$;
\item \textsf{KNN\_Hausdorff\_KL}: $k$-NN in $(S^\infty,d)$ with $d$ given in (\ref{dist}) and $d_0(x,y)=1-\exp\Big\{(\#y-\#x) \log \left( \#x / \#y\right)\Big\}$ .
\end{itemize}

In all the simulations we generated training and testing samples of size $100$ ($50$ for each class), used $100$ replications. We chose $k$ in the $k$-NN rule via cross validation. For the Bayes rule we used cross validation to get the optimal $\sigma$ in Sections \ref{smooth-case} and \ref{wiggly-case} but we fixed $\sigma = 0.1$ in Sections \ref{sec:integrenigual} and \ref{sec:nonpoisson}.
 
\subsection{Behaviour of our proposed methods in three different scenarios}
\subsubsection{Smooth case.}\label{smooth-case}
In this case, for the class 0 we generate the processes in the square $[0,1]^2$ with intensity
\begin{equation*} \label{exp1} 
\lambda_0(x,y)=c_2\exp(-20((x-1/2)^2+(y-1/2)^2)),
\end{equation*}
and for class 1,
\begin{equation*} 
\lambda_1((x,y),c_1,d_1)=c_1\exp(-d_1((x-1/2)^2+(y-1/2)^2)).
\end{equation*}
In Figure \ref{figmodiimean} we report the misclassification rate for different values of the parameters $c_1$ and $d_1$, with $c_2=500$. For a better understanding, in Figure \ref{figexp} we plot different level sets of both of the estimated intensities. Let us observe that the intensities in this case overlap considerably, which difficulties the classification. As expected, the misclassification rate decreases when the difference between $c_1-50$ and $d_1-20$ increase. 
\begin{figure}[h!]
	\begin{center}
		\includegraphics[height=9.5cm, width=1\textwidth]{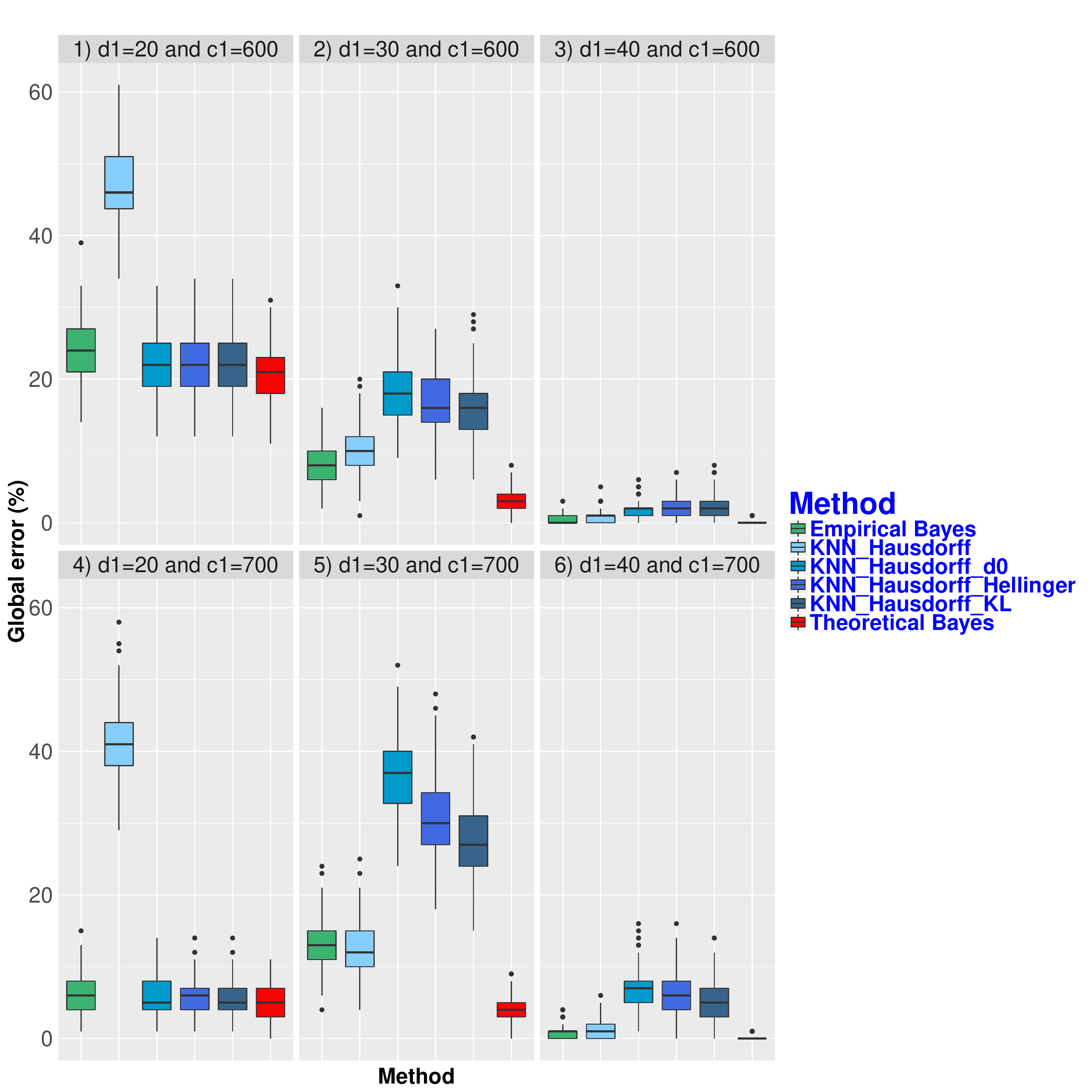} 
		\vspace{-0.85cm}
		\caption{Misclassification rates distribution for simulation from Section \ref{smooth-case}.}
		\label{figmodiimean}
	\end{center}
\end{figure}
\begin{figure}[h!]
	\begin{center}
		\includegraphics[width=0.7\textwidth]{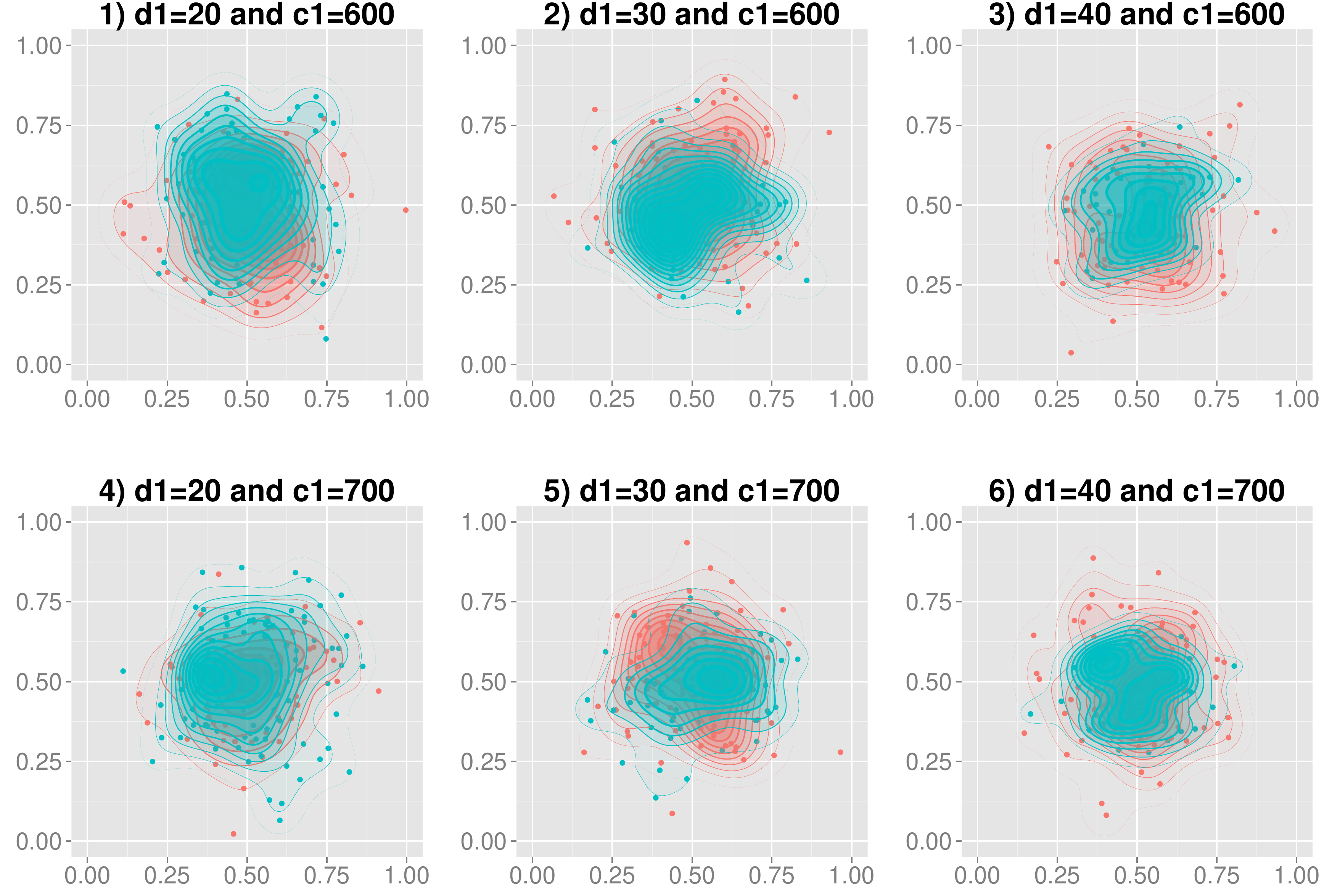} 
		\vspace{-0.4cm}
		\caption{Intensities for simulation from Section \ref{smooth-case}. }
		\label{figexp}
	\end{center}
\end{figure}

\subsubsection{Wiggly case.} \label{wiggly-case}  
In this case, for class 0 we generate the processes in the square $[0,1]^2$ with intensity
\begin{equation*} 
\lambda_0(x,y)=80+80xy\sin(1/(xy)),
\end{equation*}
and for class 1, 
\begin{equation*} 
\lambda_1((x,y),c_2)=c_2+ 30xy\sin(1/(xy)),
\end{equation*}
where $c_2$  is a positive constant. In Figure \ref{figsin} we report the boxplot of the misclassification rate for different values of the parameter $c_2$ and in Figure \ref{figw} we plot different level sets of both of the estimated intensities. 
\begin{figure}[h!]
	\begin{center}
		\includegraphics[width=0.8\textwidth]{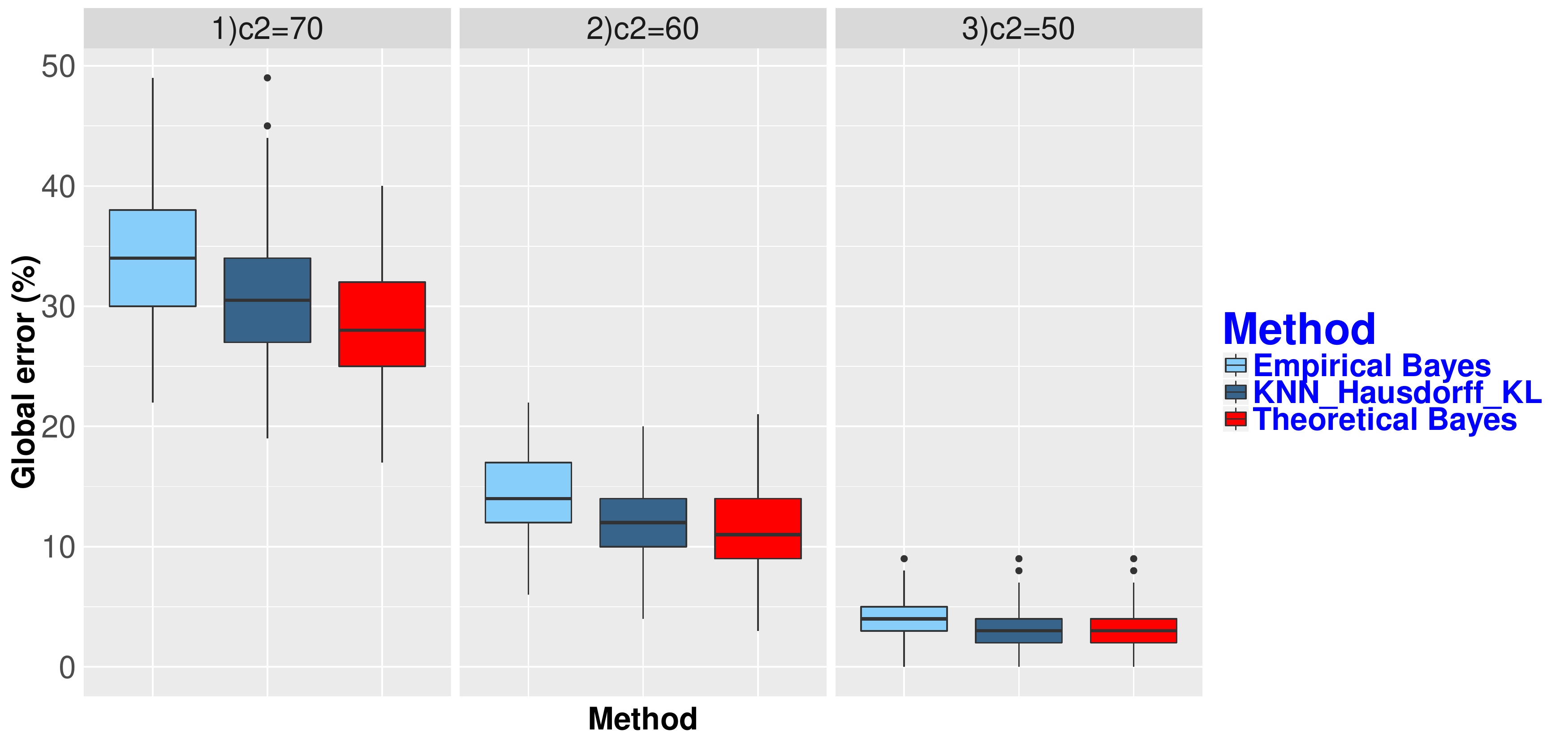} 
		\vspace{-0.4cm}
		\caption{Misclassification rates distribution for simulation from Section \ref{wiggly-case}.}
		\label{figsin}
	\end{center}
\end{figure}
\begin{figure}[h!]
	\begin{center}
		\includegraphics[width=0.8\textwidth]{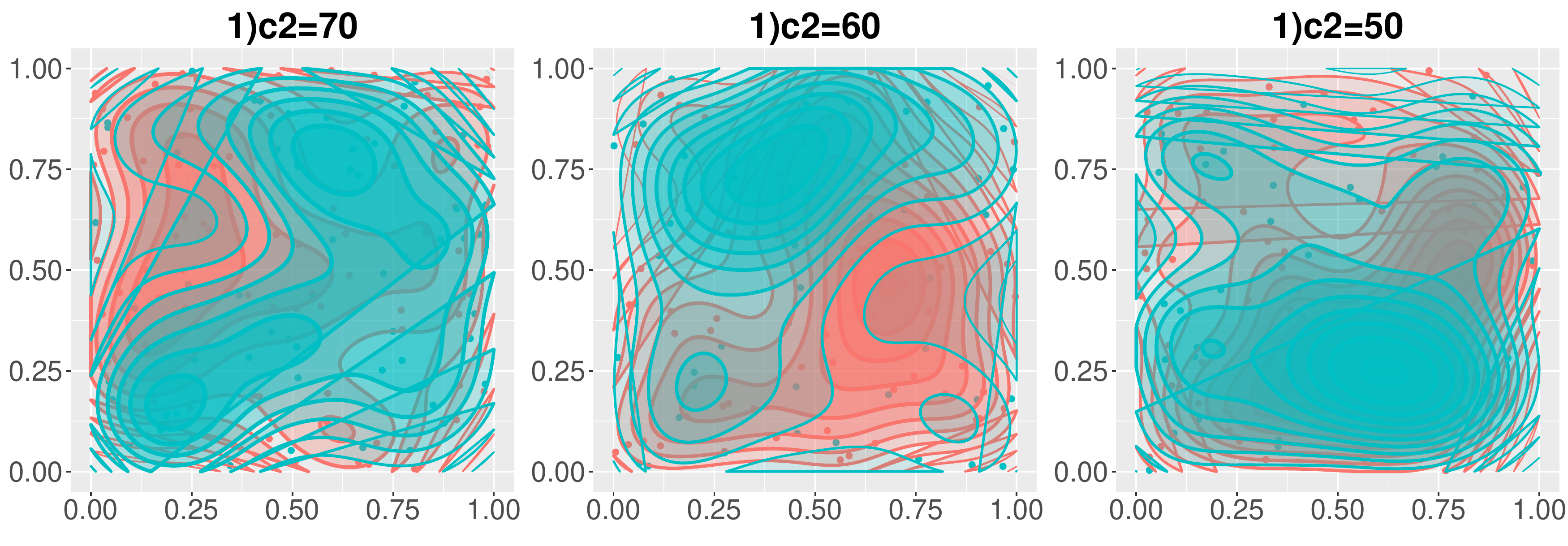} 
		\vspace{-0.4cm}
		\caption{Intensities for simulation from Section \ref{wiggly-case}.}
		\label{figw}
	\end{center}
\end{figure}

Again, as expected, the misclassification rate decreases when the difference between $c_2$ and $80$ increases.

\subsubsection{Different intensities but same expected number of points.}\label{sec:integrenigual}
In this case we generate two processes in the square $[-1,1]\times[-1,1]$. Both of them have intensity with the same height but one of them centered at $[-1/4,0]$ and the other one centered at $[0,1/4]$ as shown in Figure \ref{int_integrenigual}.
In Figure \ref{integrenigual} we report the misclassification rate where we can see that, in this case, Bayes rules performs much better than the $k$-NN rule.
\begin{figure}[h!]
	\begin{center}
		\includegraphics[width=0.8\textwidth]{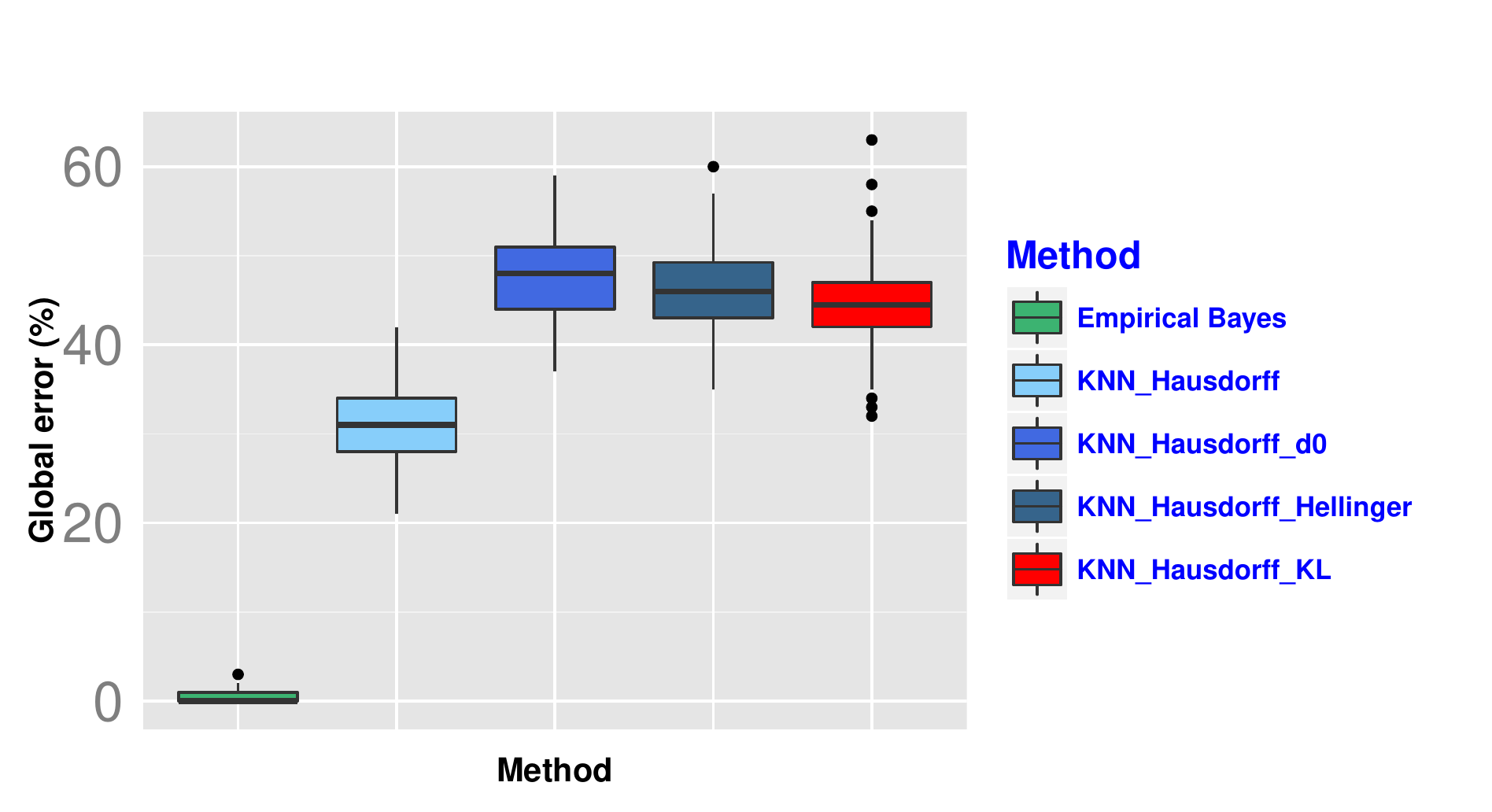} 		
		\vspace{-0.4cm}
		\caption{Misclassification rates of simulation from Section \ref{sec:integrenigual}.}
		\label{integrenigual}
	\end{center}
\end{figure}
\begin{figure}[h!]
	\begin{center}
		\includegraphics[width=0.7\textwidth]{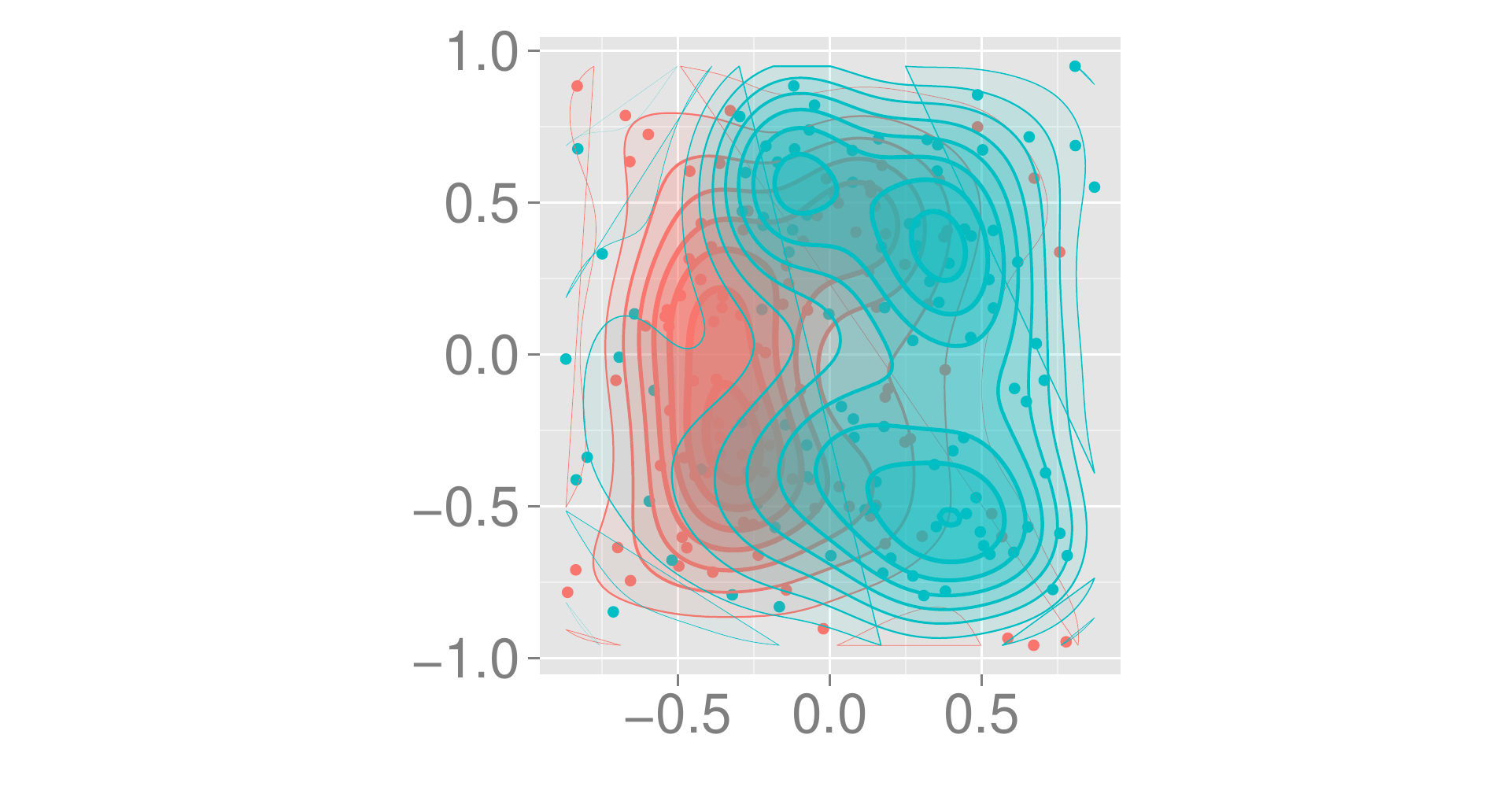} 
		\vspace{-0.4cm}
		\caption{Level sets for the estimated intensities of simulation from Section \ref{sec:integrenigual}.}
		\label{int_integrenigual}
	\end{center}
\end{figure}

As we can see in the previous simulations, in the case of densities with the same expected number of points (Section \ref{sec:integrenigual}), the estimated Bayes rules outperforms the $k$-NN based rules whereas in the wiggly case (\ref{wiggly-case}) $k$-NN achieves a better performance. This could be due to the fact that smooth intensities can be better estimated. For smooth functions (section \ref{smooth-case}) sometimes $k$-NN outperforms the Bayes rules (specially when adding an extra term to the distance).

\subsection{Robustness under non Poisson distributions} \label{sec:nonpoisson}
In this simulation we generate two Strauss processes (see Section \ref{gibss}) in the same region $W = [0, 10] \times [0, 10]$, one with parameters $\beta_1 = 0.5, \gamma_1 = 1, r_1 = 0.3$ and the other with parameters $\beta_2 = 1.5, \gamma_2 = 0.5, r_1 = 0.6$. In this case, the mean of the misclassification rates are: $0.083\%$ for the Bayes rule, $0.401\%$ for \textsf{KNN\_Hausdorff}, 
$0.072\%$ for \textsf{KNN\_Hausdorff\_d0}, $0.073\%$ for \textsf{KNN\_Hausdorff\_Hellinger} and 
$0.071\%$ for \textsf{KNN\_Hausdorff\_KL}.

\begin{figure}[h!]
	\begin{center}
		\includegraphics[width=0.8\textwidth]{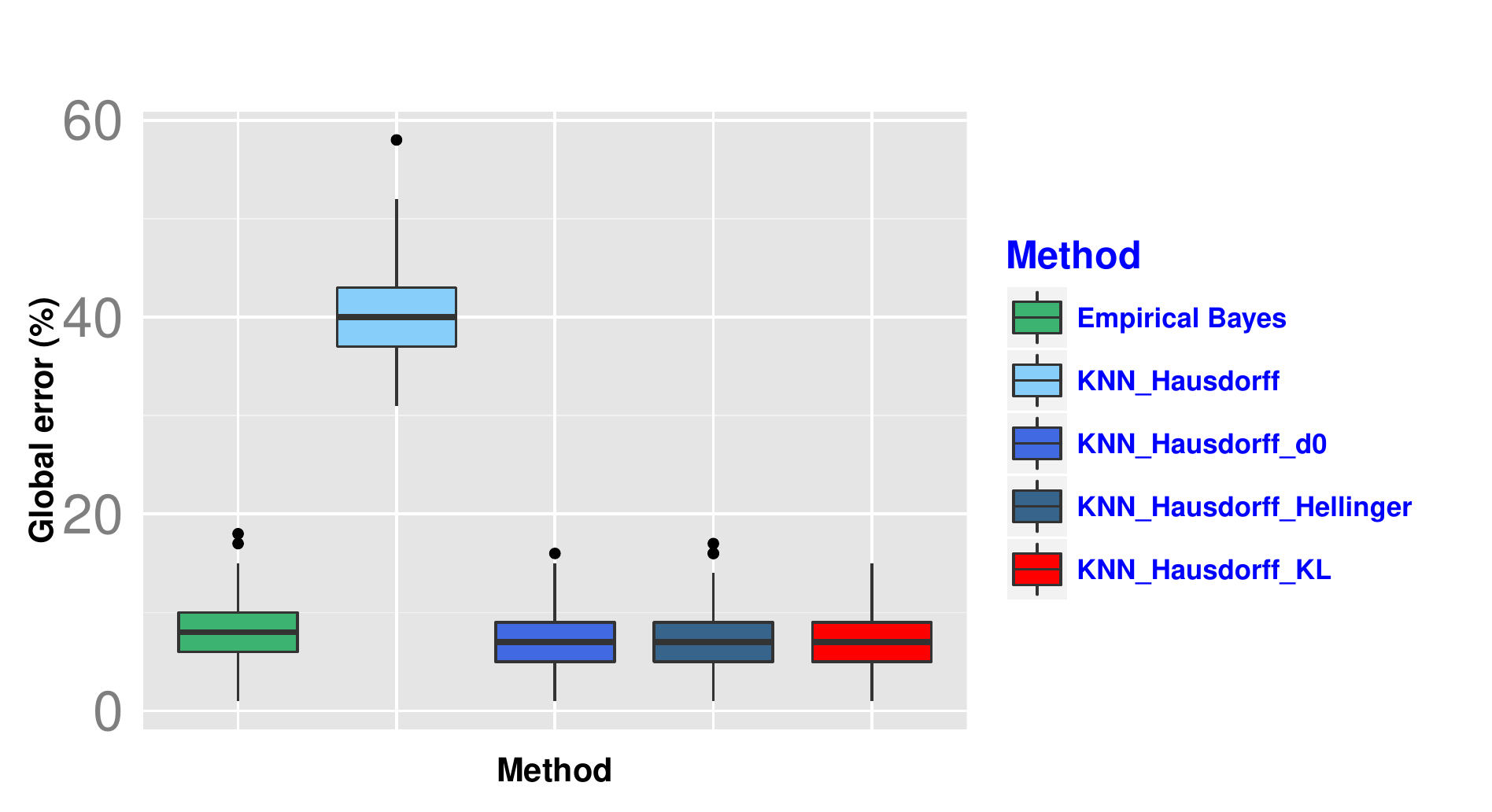} 		
		\vspace{-0.4cm}
		\caption{Misclassification rates of simulation from Section \ref{sec:nonpoisson}.}
		\label{nopoisson}
	\end{center}
\end{figure}

This shows, togheter with the boxplot of the misclassification rates (Figure \ref{nopoisson}) the robustness of our methods and a better performance for the $k$-NN base rules.

\subsection{Effect of the smoothing parameter used in the estimation of the intensities  (\ref{estint})} \label{sec:efectoh}
To show the effect of the smoothing parameter in the estimation of the intensity function (\ref{estint}), we run it in the setting described in subsection (\ref{smooth-case}) with $c_1=500$, $d_1 = 20$ and $\sigma_1$ for one of the intensities and $c_2 = 700$, $\sigma_2$ for the other. We took different combinations of $(\sigma_1,\sigma_2)$.   The misclassification rates are plotted in Figure \ref{efectoh} where in the epigraph of each graphic we put  $\sigma_2$ and in the $x$-axis $\sigma_1$. In the boxplots it can be seen that, in general, the best combination is $\sigma_1=\sigma_2$. 
\begin{figure}[h!]
	\begin{center}
		\includegraphics[width=1\textwidth]{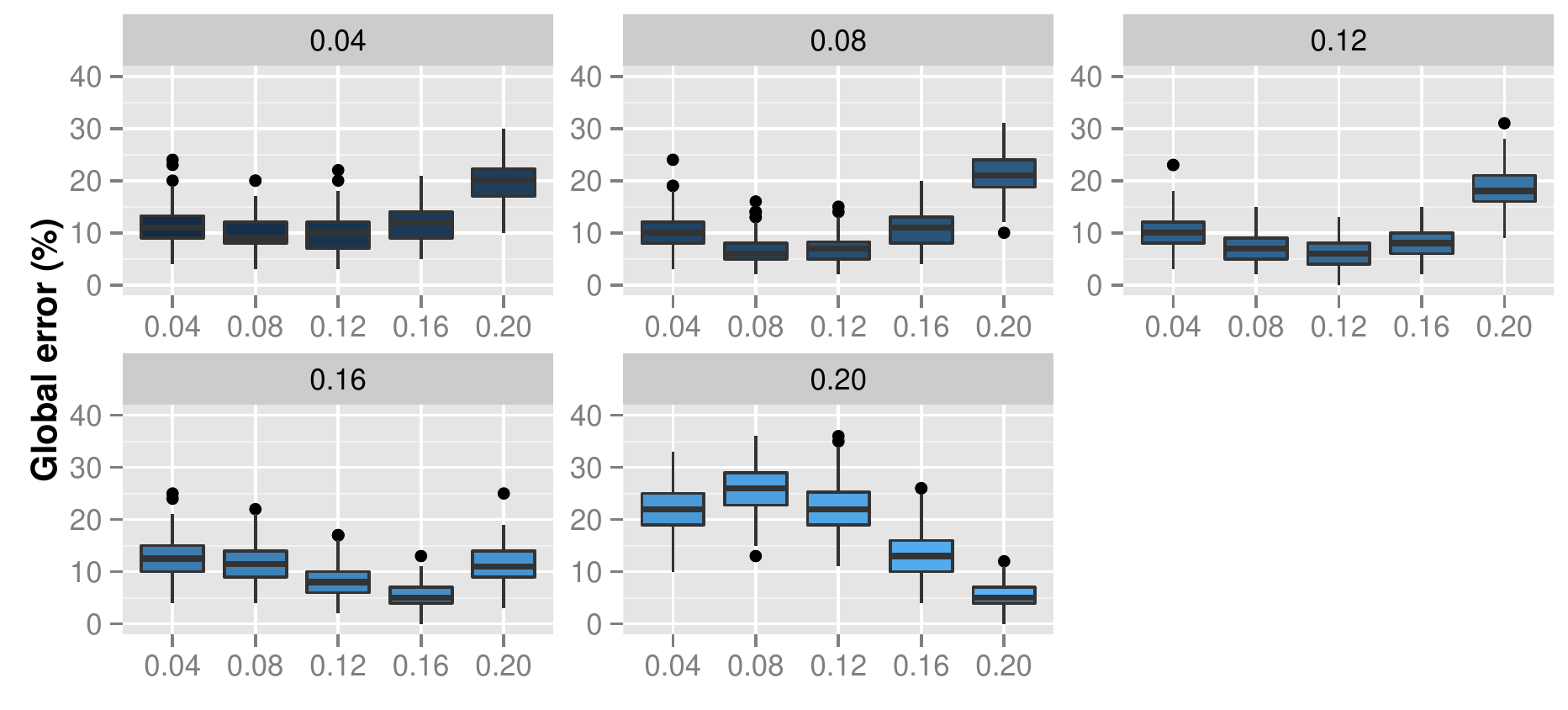} 
		\vspace{-0.4cm}
		\caption{Misclassification rates distribution of simulation from Section \ref{sec:efectoh}.}\label{efectoh}
	\end{center}
\end{figure}

\subsection{Effect of $k$ in the $k$-NN rule using different distances (\ref{dist})}\label{sec:efectok}
To asses the effect $k$ in the $k$-NN rule we run it in the setting described in subsection (\ref{smooth-case}) with $c_1=500$, $d_1 = 20$, $c_2 = 700$. The results are given in Figure \ref{efectok}. It can be seen that for all distances the $k$-NN rule performs better choosing higher values of $k$. However, it can be also seen that choosing a value greater than $7$ does not improve the performance considerably.
\begin{figure}[h!]
	\begin{center}
\includegraphics[width=0.45\textwidth]{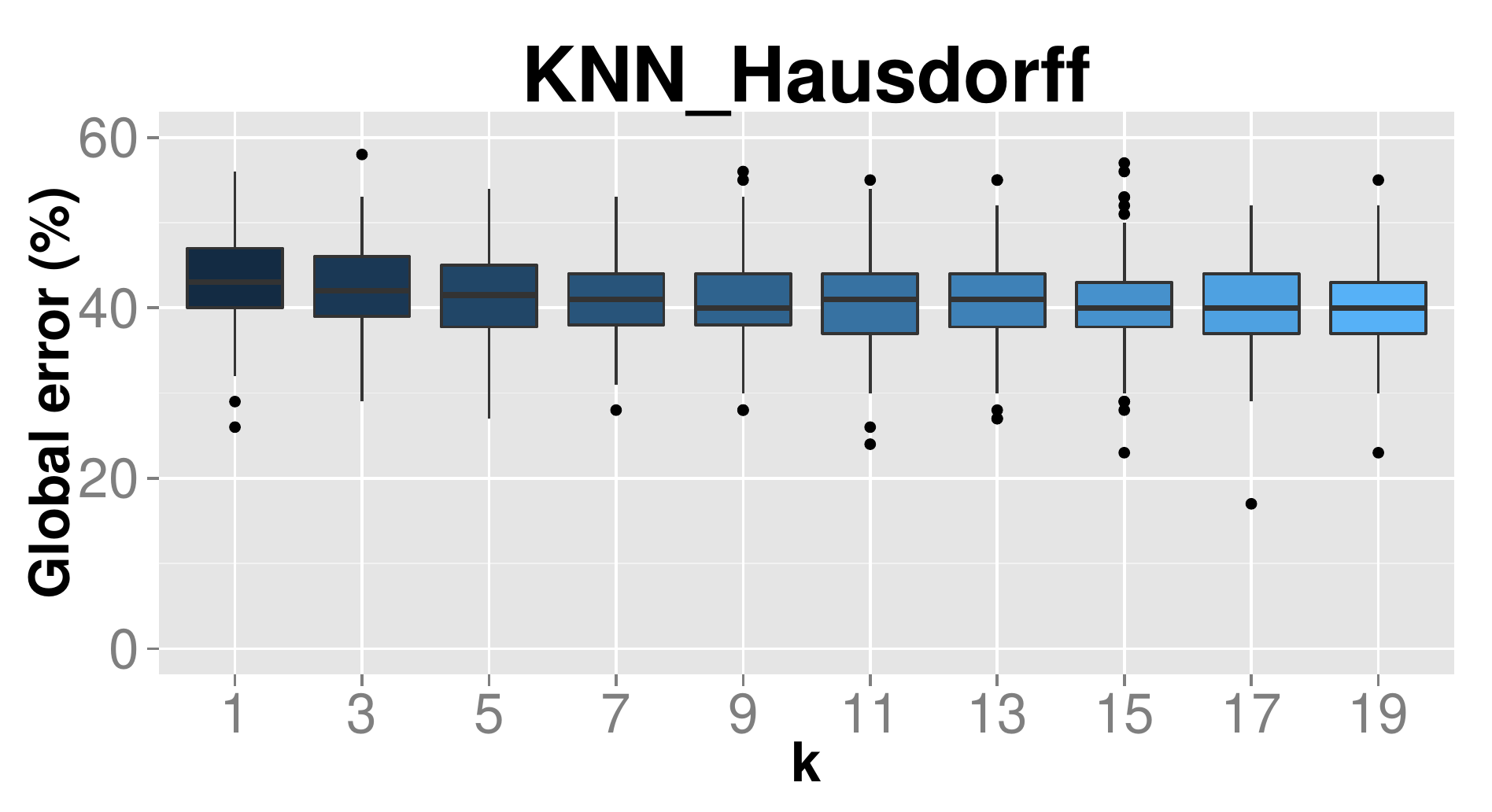} 
\includegraphics[width=0.45\textwidth]{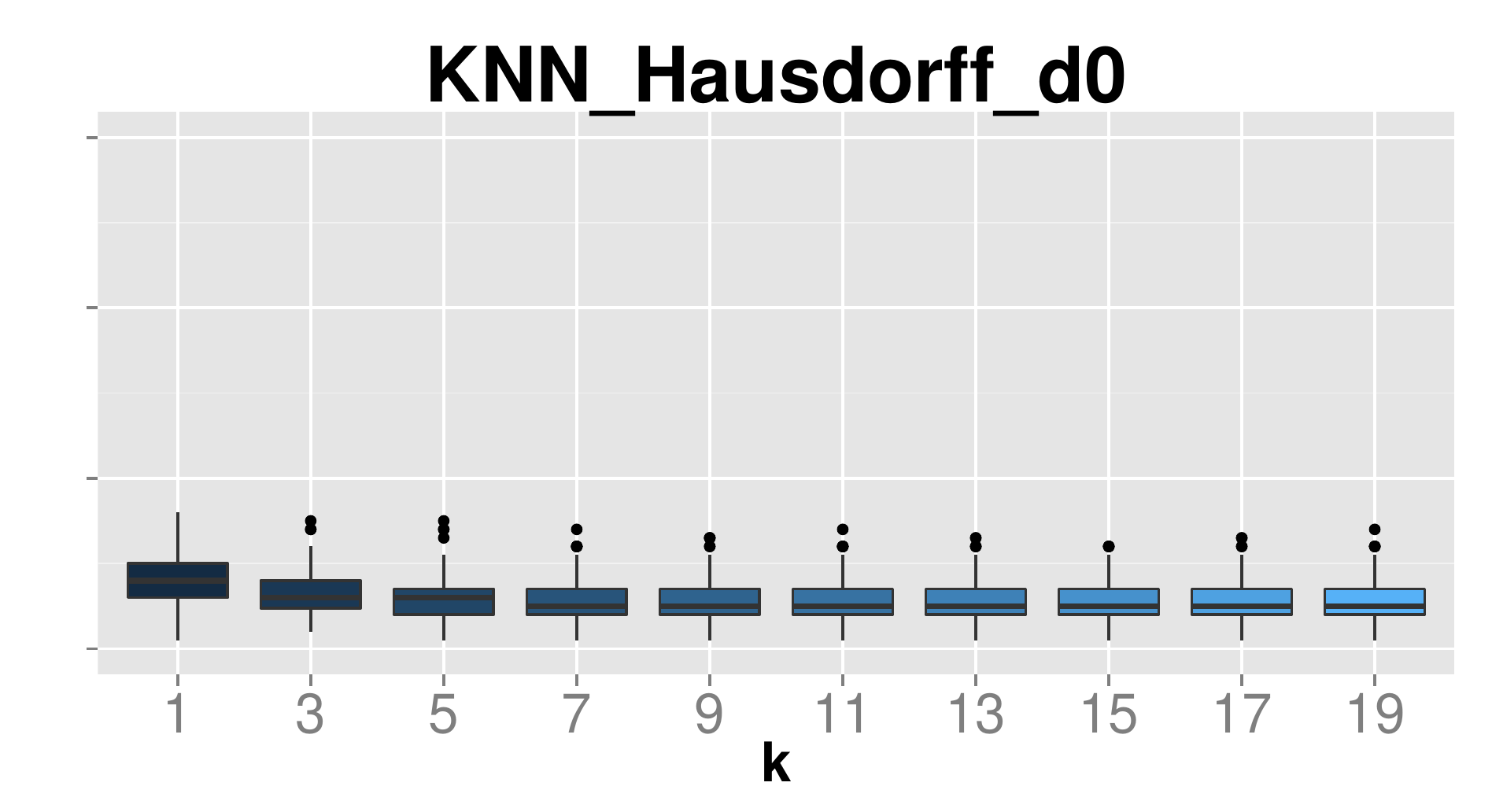}\\
\includegraphics[width=0.45\textwidth]{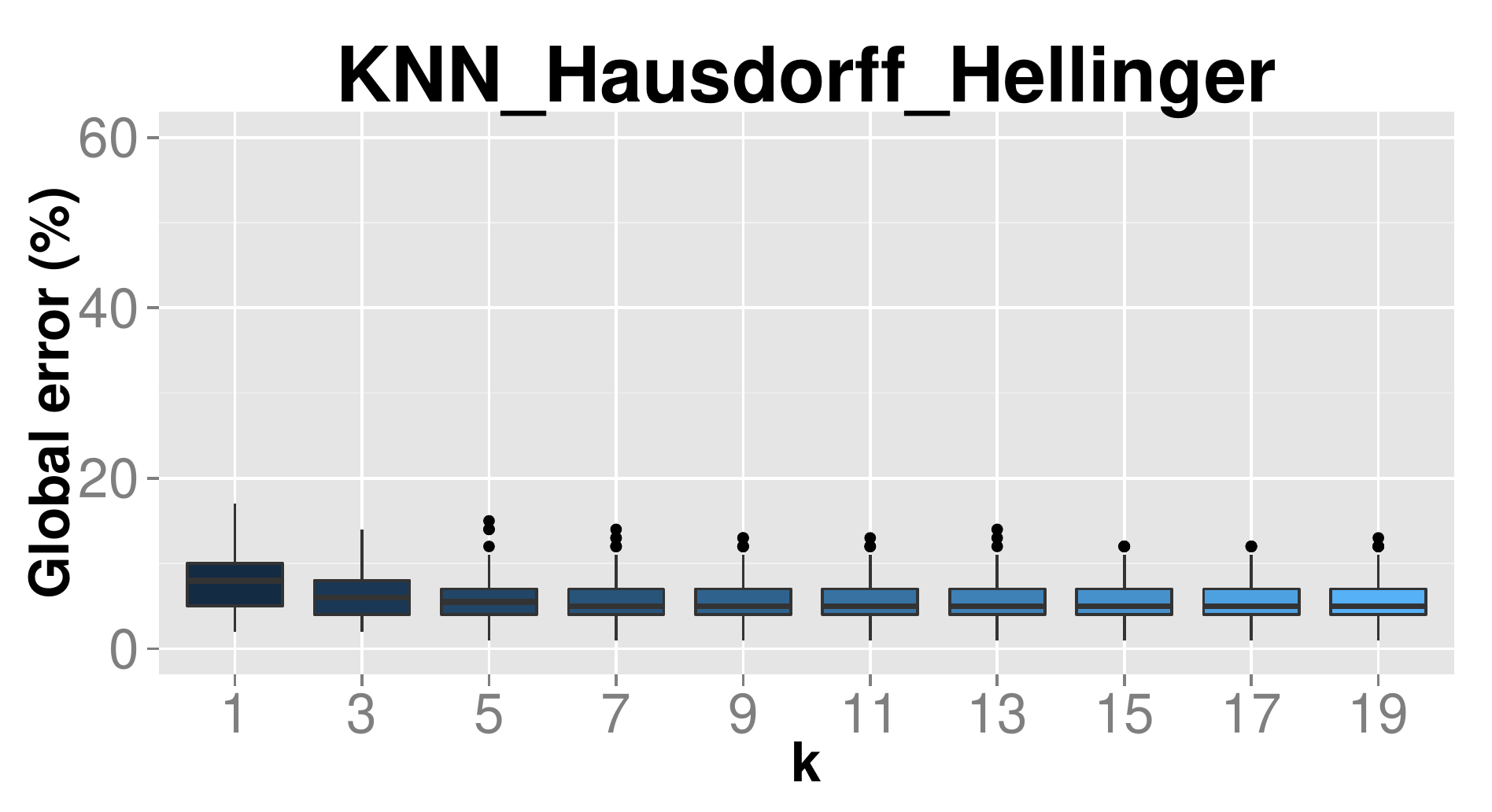} 
\includegraphics[width=0.45\textwidth]{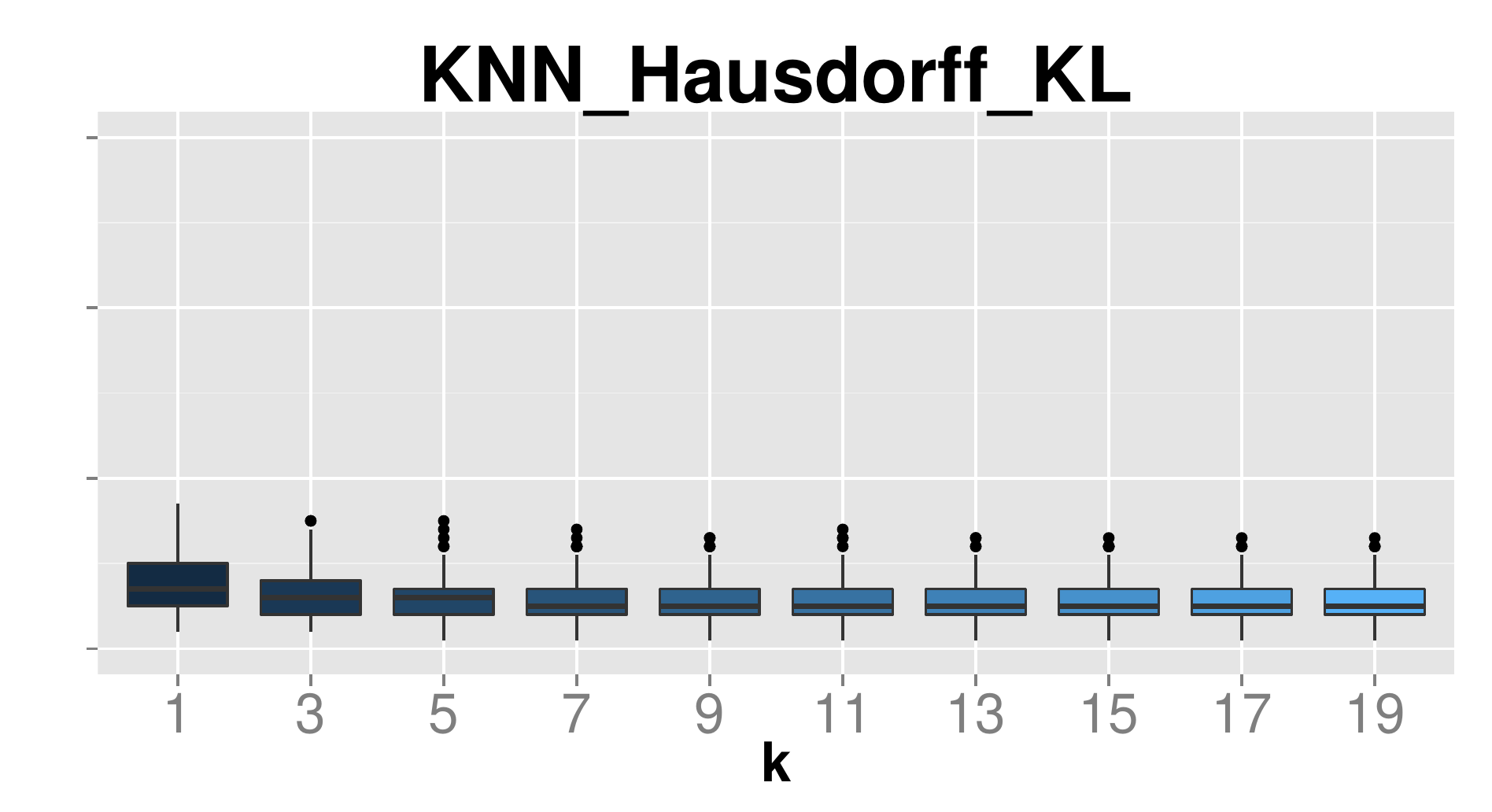}
		\vspace{-0.4cm}
		\caption{Misclassification rates distribution for the effect of $k$ in the $k$-NN rule (simulation from Section \ref{sec:efectok}). }\label{efectok}
	\end{center}
\end{figure}
\section{Real data example}\label{realdata}
The study of the geographic distribution of crimes gave rise to the well known ``social disorganization theory'', developed by the Chicago School (also called the Ecological School) which, since 1920, specializes in urban socio\-logy and urban environment research. It proposes that the neighborhood of a subject is as significant as the person's characteristics (like gender, race, etc). See Chapter 33 in \cite{handcrim} for a survey on this topic. The School collected the location as well as a description of many crimes including prostitution, assault, narcotics, battery, among others, reported between 2001 and 2016 in the city of Chicago and joint them in an open source database of more than 6 million entries. This database was recently employed in \cite{gervini} to fit a model for the intensity function of replicated point processes. In order to asses if there exists statistical differences between the spatial pattern of points of crimes, we performed  classification among the different crimes.
\begin{figure}[b!]
	\begin{center}
		\includegraphics[width=0.9\textwidth]{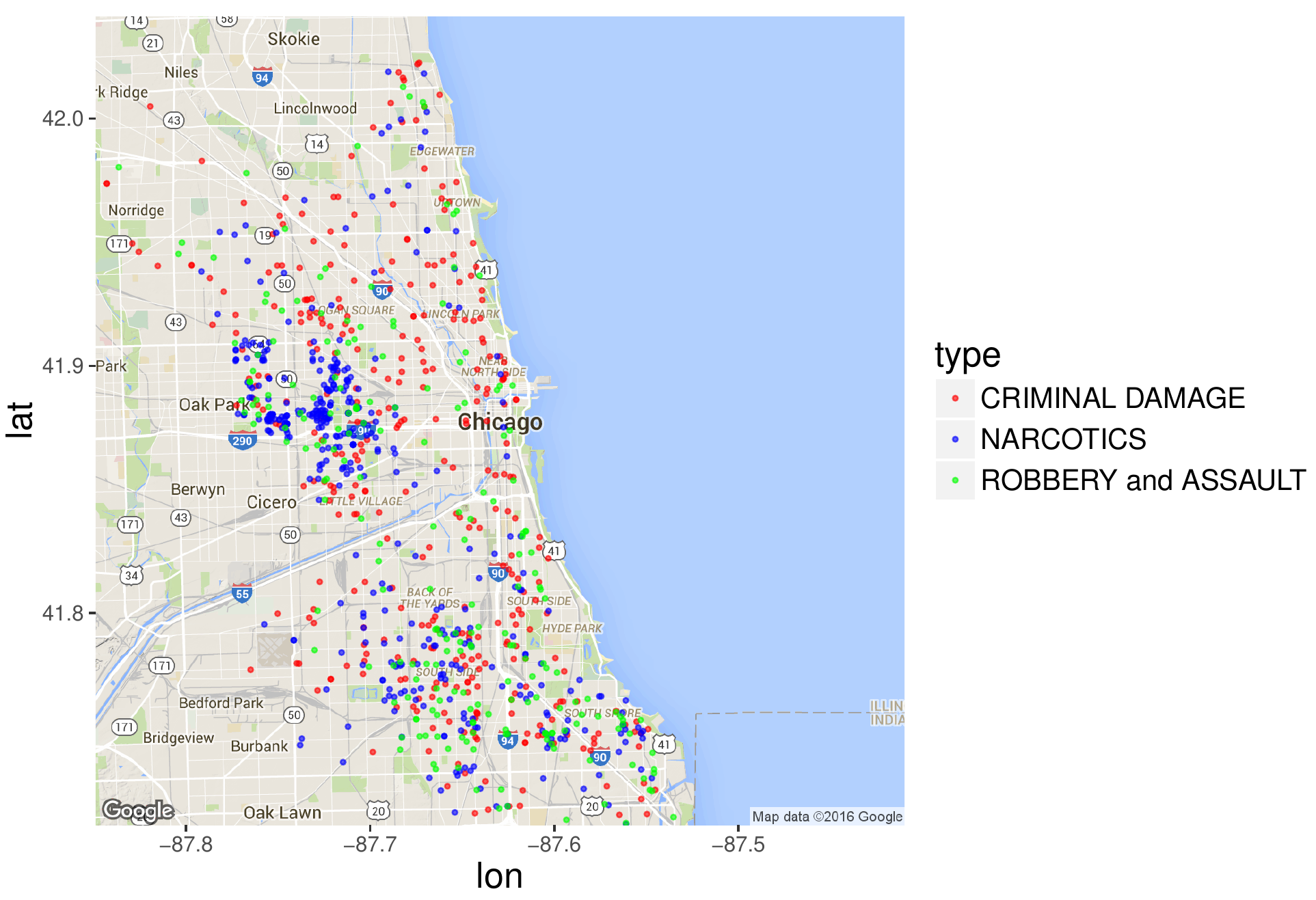} 
		\vspace{-0.4cm}
		\caption{Location of the reported assaults and robberies (green), narcotics (blue), and criminal damage (red) in one week in the city of Chicago.}
		\label{fig1}
	\end{center}
\end{figure}
 To get different samples of the same process we have split the data in periods of one week, comprehended between the first of January of 2014 and the first of January of 2016. As a result, for every type of crime we have 105 samples, 84 of them were used as training sample and the remaining 21 for the testing sample.  Since there exists a wide range of intensities between the different crimes, we have considered only three of them: \textit{assault and robber} (joined in one class, denoted as \textbf{AR}), \textit{narcotics} (\textbf{N}), and \textit{criminal damage} (\textbf{CD}). The mean value of locations registered in one week, for every sample is $502$, $482$ and $532$ for \textbf{AR}, \textbf{N}, and \textbf{CD} respectively. The classification errors obtained using $k$-NN rule for $k=20$ (this value minimize the misclassification error) were the following: between \textbf{N} and \textbf{CD} 7\%, \textbf{CD} and \textbf{AR} 6\%, and finally between \textbf{AR} and \textbf{N} we get 15\%. This result suggests that there is a stronger geographic similarity between the crimes typified as \textit{narcotics} and those typified as \textit{assault and robbery}. This can be also seen in Figure \ref{fig1}, were we represented the points for this 3 kind of crimes, reported in one week. There we can see that points in blue (\textbf{N}) and in green (\textbf{AR}) are very closed each other whereas the red ones are spread throughout all the city.

\section{Conclusions}
We have proposed two consistent classification techniques for point Poi\-sson processes: the $k$-NN and Bayes rule. The $k$-NN rule has shown better performance in cases in which the intensity function of the process is wiggly and for non Poisson processes whereas the Bayes rule did it when the intensity functions have the same expected number of points. From a theoretical point of view, we proved that the $k$-NN rule is consistent not only for the case of spatial process in $\mathbb{R}^d$, but also for processes taking values in any metric space. The rule has also shown to be robust against departures from Poisson distribution.

\section*{Appendix}\label{apendice}
\begin{proof}[Proof of Lemma \ref{bayes}]For $j=1,\ldots,M$, let  $f_{X_j}$ be the density of $X_j$ with respect to the Poisson process with intensity $1$. That is, $f_{X_j}$ is the Radon-Nykodim derivative $\frac{dP_{X_j}}{dP}(x)$ of $P_{X_j}$ with respect to the distribution $P$ of the Poisson process with intensity $\lambda=1$. Then, 
\begin{equation} \label{funcreg}
\mathbb{P}(Y=j|X=x)=\frac{f_{X_j}(x)\mathbb{P}(Y=j)}{f_{X}(x)}, \quad j=1,\ldots,M
\end{equation}
with $f_X(x)= \sum_{j=1}^M f_{X_j}(x)\mathbb{P}(Y=j)$ the total probability. Let $p_j=\mathbb{P}(Y=j)$, then we have 
\begin{equation}\label{ec1}
\mathbb{P}(Y=j|X=x) >  \mathbb{P}(Y=i|X=x) \Leftrightarrow  \frac{f_{X_j}(x)}{f_{X_i}(x)} > \frac{p_i}{p_j}.
\end{equation}
Now, since $\mu_j(S)=\int_S \lambda_j (\zeta)d\nu(\zeta)$ for $j=0,\ldots, i, \ldots,M$, from equation (\ref{eqdens2}) we get,
\begin{align*}
\frac{f_{X_j}(x)}{f_{X_i}(x)} &= \frac{\exp\Big[ \nu(S)-\mu_j(S)\Big] \prod_{\xi \in x} \lambda_j(\xi)}{\exp\Big[ \nu(S)-\mu_i(S)\Big] \prod_{\xi \in x} \lambda_i(\xi)} 
\\ &= \exp\Big[\mu_i(S)-\mu_j(S)\Big]\prod_{\xi \in x}  \frac{\lambda_j(\xi)}{ \lambda_i(\xi)}.
\end{align*}
And with this equality in (\ref{ec1}), it turns out that
\begin{align*}
& \mathbb{P}(Y=j|X=x) >  \mathbb{P}(Y=i|X=x) \nonumber \\ & \nonumber \\ & \hspace{3cm}\vertiff\\ & \exp\Big[\mu_i(S)-\mu_j(S)\Big]\prod_{\xi \in x}  \frac{\lambda_j(\xi)}{ \lambda_i(\xi)} > \frac{p_i}{p_j}.\nonumber
\end{align*}
Therefore, the Bayes rule classifies a point $x\in S^\infty$ in class $j$ if 
$$\exp\Big[\mu_i(S)-\mu_j(S)\Big]\prod_{\xi \in x}  \frac{\lambda_j(\xi)}{ \lambda_i(\xi)} > \frac{p_i}{p_j}, \text{ for all }i\neq j.$$
\end{proof}
 
\begin{proof}[Proof of Theorem \ref{consistencia}]
Let us fix $\zeta\in supp(\nu)$, and write,
\begin{equation} \label{descomp}
\Big|\hat{\hat{\lambda}}_m(\zeta)-\lambda(\zeta)\Big|\leq \Big|\hat{\hat{\lambda}}_m(\zeta)-\mathbb{E}\big(\hat{\hat{\lambda}}_m(\zeta)\big)\Big|+\Big|\mathbb{E}\big(\hat{\hat{\lambda}}_m(\zeta)\big)-\lambda(\zeta)\Big|.
\end{equation}
First observe that, conditioned to $\# X_j = n(j)$, the random variables $\xi_1,\dots,\xi_{n(l)}$ are an iid sample of $\xi$ with density $\lambda(\zeta)/\mu(S)$ (see Def 3.2 Moller-Waagepetersen), where $\mu(S)=\int_S \lambda(\xi)d\nu(\xi)$. In addition, since $X$ is a Poisson process, $\# X_j \sim \mathcal{P}(\mu(S))$, so that $\mathbb{E}(\# X_j) = \mu(S)$. Then, 
\begin{align*}
\mathbb{E}\big(\hat{\hat{\lambda}}_m(\zeta)\big) &= \mathbb{E}\Big[\mathbb{E}\Big(\hat{\hat{\lambda}}_m(\zeta)\Big|\# X_1=n(1),\dots,\#X_m=n(m)\Big)\Big]
\\ &= \mathbb{E}\Big[\mathbb{E}\Big(\frac{1}{m}\sum_{j=1}^m \frac{1}{K_{\sigma_m}(\zeta)}\sum_{i=1}^{n(j)}  k_{\sigma_m}\big(\rho(\zeta,\xi_i)\big)\Big|\# X_1=n(1),\dots,\#X_m=n(m)\Big)\Big] \\ &= \mathbb{E}\Big[\frac{1}{m}\sum_{j=1}^m \frac{1}{K_{\sigma_m}(\zeta)}\sum_{i=1}^{\# X_j} \mathbb{E}\big( k_{\sigma_m}\big(\rho(\zeta,\xi_i)\big)\Big|\# X_1=n(1),\dots,\#X_m=n(m)\Big)\Big] \\ &= \mathbb{E}\Big[\frac{1}{K_{\sigma_m}(\zeta)}\frac{1}{m}\sum_{j=1}^m \sum_{i=1}^{\# X_j} \frac{1}{\mu(S)} \int_S k_{\sigma_m}(\rho(\zeta,\xi))\lambda(\xi)d\nu(\xi))\Big] \\ &= \frac{1}{K_{\sigma_m}(\zeta)\mu(S)}  \int_S k_{\sigma_m}(\rho(\zeta,\xi))\lambda(\xi)d\nu(\xi) \mathbb{E}\Big[\frac{1}{m}\sum_{j=1}^m \# X_j \Big]  \\ &= \frac{1}{K_{\sigma_m}(\zeta)}  \int_S k_{\sigma_m}(\rho(\zeta,\xi))\lambda(\xi)d\nu(\xi))  \\
&=\mathbb{E}(\hat{\lambda}_1(\zeta)).
\end{align*}
With this in (\ref{descomp}) we have,
\begin{equation} \label{IyII}
\Big|\hat{\hat{\lambda}}_m(\zeta)-\lambda(\zeta)\Big|\leq \Big|\hat{\hat{\lambda}}_m(\zeta)-\mathbb{E}(\hat{\lambda}_1(\zeta))\Big|+\Big|\mathbb{E}(\hat{\lambda}_1(\zeta))-\lambda(x)\Big|\doteq I+II.
\end{equation}
To prove that $I \rightarrow 0$ observe that,
\begin{align}\label{uno}
\hspace{-0.5cm} & \mathbb{P}\left(\Bigg|\frac{1}{m}\sum_{j=1}^m \hat{\lambda}_j(\zeta)-\mathbb{E}(\hat{\lambda}_1(\zeta))\Bigg|>\epsilon\right)\\ &=
\mathbb{E}\left[\mathbb{P}\left(\Bigg|\frac{1}{m}\sum_{j=1}^m \hat{\lambda}_j(\zeta)-\mathbb{E}(\hat{\lambda}_1(\zeta))\Bigg|>\epsilon \, \Bigg|\# X_1=n(1),\dots,\#X_m=n(m)\right)\right]. \nonumber
\end{align}
In order to apply Hoeffding inequality to $\hat{\lambda}_j(\zeta)$ observe that, conditioned to $\#X_j=n(j)$, if we denote $\gamma_m(\zeta)=\nu(B_\rho(\zeta,\sigma_m))$,  $0\le \hat{\lambda}_j(\zeta) \leq K_1n(j)/\gamma_m(\zeta)$, $\forall x\in S$ and $j=1,\dots,m$, with $K_1= \max k(\zeta)/k_0$. Therefore, applying Hoeffding inequality we get
\begin{small}
\begin{align}\label{dos}
&\mathbb{E}\left[\mathbb{P}\left(\Bigg|\frac{1}{m}\sum_{j=1}^m \hat{\lambda}_j(\zeta)-\mathbb{E}(\hat{\lambda}_1(\zeta))\Bigg|>\epsilon \, \Bigg|\# X_1=n(1),\dots,\#X_m=n(m)\right)\right] \nonumber \\ &\le \mathbb{E}\left[2\exp\left(-\frac{2\epsilon^2(m\gamma_m(\zeta))^2}{K_1^2 \sum_{j=1}^m (\#X_j)^2} \right)\right]  \nonumber \\ 
&= \mathbb{E}\left[\mathbb{E}\bigg[2\exp\Big(-\frac{2\epsilon^2(m\gamma_m(\zeta))^2}{K_1^2(\sum_{j=1}^{m-1} (\#X_j)^2+(\#X_m)^2)} \Big)\Big|\# X_1=n(1),\dots,\#X_{m-1}=n(m-1)\bigg]\right]  \nonumber \\
&\leq \mathbb{E}\left[2\exp\left(-\frac{2\epsilon^2(m\gamma_m(\zeta))^2}{K_1^2\big(\mathbb{E}\big[\sum_{j=1}^{m-1} (\#X_j)^2+(\#X_m)^2\big|\# X_1=n(1),\dots,\#X_{m-1}=n(m-1)\big]\big)}\right)\right]  \nonumber\\
&=\mathbb{E}\left[2\exp\left(-\frac{2\epsilon^2(m\gamma_m(\zeta))^2}{K_1^2\big(\sum_{j=1}^{m-1} (\#X_j)^2+\mathbb{E}((\#X_m)^2)\big)}\right)\right]  \\
&=\mathbb{E}\left[2\exp\left(-\frac{2\epsilon^2(m\gamma_m(\zeta))^2}{K_1^2\big(\sum_{j=1}^{m-1} (\#X_j)^2+\mathbb{E}((\#X_m)^2)\big)}\right)\right]  \nonumber \\ & \hspace{0.2cm}\vdots \nonumber \\ &\le 
2\exp\left(-\frac{2\epsilon^2(m\gamma_m(\zeta))^2}{K_1^2 \sum_{j=1}^m \mathbb{E}((\#X_j)^2)} \right)  \nonumber\\ &= 2\exp\left(-\frac{2\epsilon^2m\gamma_m(\zeta)^2}{K_1^2 \mu(S)\big(1+\mu(S)\big)} \right), \nonumber
\end{align}
\end{small}
where we used the same conditioning trick $m$ times and, in the last equality, we used that $\# X_j \sim \mathcal{P}(\mu(S))$ so that $\text{var}(\# X_j) = \mu(S)$. 
Now, by Lemma $A2$ in \cite{ffl:12} (with $k_m=\log(m)^2$), there exists $\sigma_m(\zeta)$ such that $\gamma_m(\zeta) \ge \log(m)/\sqrt{m}$. Therefore, 
\begin{align}\label{tres}
\sum_{m=1}^\infty 2\exp\left(-\frac{2\epsilon^2m\gamma_m(\zeta)^2}{K_1^2\mu(S)\big(1+\mu(S)\big)} \right) &\leq \sum_{m=1}^\infty 2\exp\left(-\frac{2\epsilon^2 \log(m)^2}{K_1^2\mu(S)\big(1+\mu(S)\big)} \right) \nonumber \\ &= 2\sum_{m=1}^\infty m^{
-\frac{2\epsilon^2\log(m) }{K_1^2\mu(S)\big(1+\mu(S)\big)}}  < \infty,
\end{align} 
then, from (\ref{uno}), (\ref{dos}) and (\ref{tres}) it follows that,
\[
\sum_{m=1}^\infty  \mathbb{P}\left(\Bigg|\frac{1}{m}\sum_{j=1}^m \hat{\lambda}_j(\zeta)-\mathbb{E}(\hat{\lambda}_1(\zeta))\Bigg|>\epsilon\right) < \infty,
\]
and finally, by Borel-Cantelli's Lemma,
$$I=\Big|\frac{1}{m}\sum_{j=1}^m \hat{\lambda}_j(\zeta)-\mathbb{E}(\hat{\lambda}_1(\zeta))\Big|\rightarrow 0, \quad {a.s}.$$
In order to prove that $II \to 0$ in (\ref{IyII}) observe that, since $\lambda$ is a continuous function and $S$ is compact, for $\epsilon >0$ there exists $m_0$, such that for all $x\in S$, $\sup_{\xi \in B(\zeta,\sigma_m)}|\lambda(\xi)-\lambda(\zeta)|<\epsilon$ if $m>m_0$. In addition, $\frac{1}{K_{\sigma_m}(x) }\int_{B(\zeta,\sigma_m)} k_{\sigma_m}(\rho(\zeta,\xi))d\nu(\xi)=1$ then we get, 
\begin{align*}
II = |\mathbb{E}(\hat{\lambda}_1(\zeta))-\lambda(\zeta)| & =  \Big| \frac{1}{K_{\sigma_m}(\zeta) }\int_{B(\zeta,\sigma_m)} k_{\sigma_m}(\rho(\zeta,\xi))\lambda(\xi)d\nu(\xi)-\\ 
&\hspace{2cm} \lambda(x)\frac{1}{K_{\sigma_m}(\zeta) }\int_{B(\zeta,\sigma_m)} k_{\sigma_m}(\rho(\zeta,\xi))d\nu(\xi)\Big|\\
&\leq  \frac{1}{K_{\sigma_m}(\zeta) }\int_{B(\zeta,\sigma_m)} k_{\sigma_m}(\rho(\zeta,\xi))|\lambda(\xi)-\lambda(\zeta)|d\nu(\xi)\\
&< \epsilon,
\end{align*}
which completes the proof. 
\end{proof}
\begin{proof}[Proof of Proposition \ref{separ}] It follows directly from the separability of the space of compact subsets of $S$, endowed with the distance $d_H$.
\end{proof}
\begin{proof} [Proof of Proposition \ref{besifuerte}]
We will make use of the following Lemma:
\begin{lemma}\label{extra} Under the hypothesis of Proposition \ref{besifuerte}, for $x\in S^\infty$ and $\epsilon>0$ there exists $\delta =\delta(x,\epsilon)$ such that \begin{equation}\label{1} 
|\eta (x) - \eta (y) | < \epsilon/2 
\end{equation}
whenever $\# y = \#x$ and $y \in B_{d_H} (x,r)$, for all $r\le \delta$. 
\begin{proof} Observe that, from equality (\ref{funcreg}), to prove (\ref{1}) is enough to prove that for every $x=\{\xi_1,\dots,\xi_k\}\in S^\infty$ there exists $\delta=\delta(x)$ such that for all $y=\{\theta_1,\dots,\theta_k\}\in B_{d_H}(x,\delta)$, $|f_{X_i}(x)-f_{X_i}(y)|<\epsilon$ for $i=0,1$ (where $B_{d_H}(x,\delta)$ denotes the closed ball of radii $\delta$ in $S^\infty$, with the distance $d_H$). From identity (\ref{eqdens2}) we have, 
$$
|f_{X_i}(x)-f_{X_i}(y)|=\exp\Big[ \nu(S)-\mu_i(S)\Big]\Big|\prod_{i=1}^k\lambda_i(\xi_i)-\prod_{i=1}^k\lambda_i(\theta_i)\Big|.
$$
Therefore, since $\lambda_i$ is continuous, for all $\epsilon>0$ there exists $\delta=\delta(k,\epsilon)$ such that $|f_{X_i}(x)-f_{X_i}(y)|<\epsilon$ and the Lemma is proved.
\end{proof}
\end{lemma}
Now let us prove the Proposition. Since $|\eta|\le 1 $, for $r<\delta$ from Lemma \ref{extra},
\begin{align*} 
 \int_{B_{d_H} (x,r)} | \eta(y) - \eta(x) | dP_X (y) &\le \frac{ \epsilon}{2} P_X (B_{d_H} (x,r)) \\  &\hspace{1cm}+   \int_{B_{d_H} (x,r)} | \eta(y) - \eta(x) | \mathbb{I}_{\{\# y\not =\#x\}}dP_X (y) \nonumber\\
 &\le \frac{ \epsilon}{2} P_X (B_{d_H} (x,r)) \\  &\hspace{1cm}+   2   \int_{B_{d_H} (x,r)}  \mathbb{I}_{\{\# y\not =\#x\}}dP_X (y). \nonumber
\end{align*}
Now, if $x=(\xi_1,\dots,\xi_{s_x})$, taking $r<r_1 (x,s_x)=\min_{i\neq j}\rho(\xi_i,\xi_j)/2$ the balls $B_\rho (\xi_i,r) $ are disjoint. Let us take
$y \in B_{d_H}(x,r)$ with $y=\{\theta_1,\dots,\theta_{s_y}\}$, $s_y \neq s_x$. 
Suppose there exists $j=1,\dots,s_y$ such that for all $i=1,\dots,s_x$, $d_\rho (\xi_i, \theta_j) \ge r$, then $d_H (x,y) \ge r$ that is a contradiction. Therefore
for each $j=1,\dots,s_y $ there exists $i=1,\dots,s_x$ with $d_\rho (\xi_i,\theta_j) < r $ and as a consequence
$\#(y\cap B_\rho(\xi_i,r))\geq 1$ for all $i=1,\dots,s_x$. Since $B_\rho (\xi_i,r) $  are disjoint we get $\#y \ge \#x$.  
As a consequence,
\begin{align*}
\int_{B_{d_H} (x,r)} | \eta(y) - \eta(x) | dP_X (y)&\le\frac{ \epsilon}{2} P_X (B_{d_H} (x,r)) \\ &\hspace{1cm}+  2   P_X \left(  {B_{d_H} (x,r)}   \cap {\{\# y  >\#x\}} \right).
\end{align*}
Observe that, 
\begin{eqnarray*}
{B_{d_H} (x,r)}   \cap {\{y: \# y  >\#x\}} & =&\{ y: \exists  i=1,\dots,s_x \hbox{  with } \# ( B_\rho (\xi_i,r) \cap y) > 1 \}\\
& \subset &\cup_{i=1}^{s_x} \{ y:   \# ( B_\rho (\xi_i,r) \cap y) > 1 \}. 
\end{eqnarray*}
Using that $N_{ B_\rho (x_i, r)}$ is a real valued random variable with Poisson distribution with parameter $ \mu (B_\rho (x_i, r))$ we have
\begin{eqnarray*}
P_X (\{ y:   \# ( B_\rho (\xi_i,r) \cap y) > 1 \}) &=& P_X ( N_{ B_\rho (\xi_i, r)} >1)\\
&\le & \frac{1}{2} \mu^2 (B_\rho (\xi_i,r)).
\end{eqnarray*}
Then, since the balls $\{B_\rho (\xi_i,r)\}_{i=1,\dots,s_x}$ are mutually disjoint, we get
\begin{eqnarray}\label{otra}
 P_X \left(  {B_{d_H} (x,r)}   \cap {\{\# y  >\#x\}} \right) &\le & \frac{1}{2}\sum_{i=1}^{s_x} \mu^2 (B_\rho (\xi_i,r)).
 \end{eqnarray}
Finally, since $\nu$ does not have atoms and $\lambda$ is locally integrable, by Radon-Nikodym's theorem, we have that for all $\epsilon>0$ there exists $r_2=r_2 (x,s_x)$ such that, if $r \le r_2 (x,s_x)$,
\[ 
\frac{1}{2}\sum_{i=1}^{s_x} \mu^2 (B_\rho (\xi_i,r)) < \epsilon/2.
\]
Taking $ r \le \min \{ r_1, r_2\} $ we get using (\ref{otra}) that
\begin{eqnarray*}
\int_{B_{d_H} (x,r)} | \eta(y) - \eta(x) | dP_X (y) &\le &\epsilon P_X (B_{d_H} (x,r)).
\end{eqnarray*}
\end{proof}
\begin{proof}[Proof of Proposition \ref{separ2}] 
The fact that $(S^\infty,d)$ is separable is a direct consequence of Proposition \ref{separ}. 
\end{proof}
\begin{proof}[Proof of Proposition \ref{besi2}] 
Let us take a point $x\in S^\infty$, since $S$ is bounded we know that $\# x=k<\infty$. Let us denote $x=\{\xi_1,\dots,\xi_k\}$, if $y\in B_d(x,\epsilon_0)$, (being $\epsilon_0$ as in Definition \ref{dist} 4.), then $d_0(x,y)<\epsilon_0$ 
and then $\#y=\#x$. As a consequence, for all $y\in B_d(x,\epsilon_0)$, if we denote $y=\{\theta_1,\dots,\theta_k\}$ (where $\theta_i$ is the nearest point to $\xi_i$ with respect to $\rho$),
$$|f(x)-f(y)|=\exp\Big[ \nu(S)-\mu_1(S)\Big]\Big|\prod_{i=1}^k\lambda(\xi_i)-\prod_{i=1}^k\lambda(\theta_i)\Big|.$$
Since $\lambda$ is continuous for all $\epsilon>0$ there exists $\delta=\delta(k,\epsilon)$ such that for all $y\in B_d(x,\delta)$, $|f(x)-f(y)|<\epsilon$. Now the continuity of $\eta$ follows from the continuity of $f$ and equation (\ref{funcreg}).
\end{proof}
\begin{proof} [Proof of Proposition \ref{diminfiprop}]
We will prove that, if there exists $r_0>0$ such that for all $\zeta\in S$, and all $r<r_0$, $\partial B(\zeta,r)$ contains two points  $\pi^1,\pi^2$ such that $\rho(\pi_1,\pi_2)=2r$ and therefore the space $(S^\infty,d\big)$ is not finite dimensional. Without lost of generality let us assume that $diam(S)=1$. It is enough to find, for all $n>0$ a point $x\in S^\infty$ and a positive number $t$ such that $B_d(x,t)$ contains $n$ points $\{y_1,\dots,y_n\}$ in $S^\infty$ fulfilling the condition $d(y_i,y_j)>t$ for all $i\neq j$. Let us take $n$ different points $\{\xi_1,\dots,\xi_n\}\in S$. We define $x=\{\xi_1,\dots,\xi_n\}$ and $t=\min\big\{\min_{i\neq j}\rho(\xi_i,\xi_j),r_0\big\}$. 
For all $i=1,\dots,n$, there exists $\pi_i^1,\pi_i^2$ different points in $\partial B(\xi_i,2t/3)$ such that $\rho(\pi_i^1,\pi_i^2)=4t/3$.
We define 
\begin{align*}
y_1=& \{\pi_1^1,\dots,\pi_n^1\},\\
y_i=& \{\pi_1^1,\dots,\pi_{i-1}^1,\pi_i^2,\pi_{i+1}^1,\dots,\pi_n^1\} \quad i=2,\dots,n.
\end{align*}
Then $d(y_i,y_j)=\rho(\pi_i^2,\pi_i^1)=4t/3>t$ for all $i\neq j$ and $y_i\in B_d(x,t)$ for $i=1,\dots,n$.
\end{proof}

\begin{proof}[Sketch of proof of Proposition \ref{besifuerte2}]. 
Since  $U$ is a continuous function, it is easy to see that the result in Lemma \ref{extra} still holds, and then
\begin{equation*}
|\eta (x) - \eta (y) | < \epsilon/2, 
\end{equation*}
whenever $\# y = \#x$ and $y \in B_{d_H} (x,r)$, for all $r\le \delta$. Now proceeding as in the proof of Proposition \ref{besifuerte}, we can write,
\begin{align*}
\int_{B_{d_H} (x,r)} | \eta(y) - \eta(x) | dP_X (y)&\le\frac{ \epsilon}{2} P_X (B_{d_H} (x,r)) \\ &\hspace{1cm}+  2   P_X \left(  {B_{d_H} (x,r)}   \cap {\{\# y  >\#x\}} \right).
\end{align*}
Let us recall first that the probability of a configuration of $n$ points in $B\subset S$ is,
\begin{equation*}
\mathbb{P}(\#(X\cap B)=n)=\frac{e^{-\nu(S)}}{n!}\int_{B^n}f(x)d\nu(x_1)\dots d\nu(x_n).
\end{equation*}
From where it follows that
		$$P_X ( N_{ B_\rho (\xi_i, r)} >1)\leq \frac{f_1}{2} \nu^2 (B_\rho (\xi_i,r)),$$ 
		being $f_1=\max_{x\in S} f(x)$. Then the rest of the proofs follows bye using the same ideas as in Proposition \ref{besifuerte}.
	\end{proof}

\end{document}